\documentclass[11pt]{amsart}
\usepackage{amssymb,mathrsfs,graphicx,enumerate}
\usepackage{amsmath,amsfonts,amssymb,amscd,amsthm,bbm}
\usepackage[retainorgcmds]{IEEEtrantools}
\usepackage{colortbl}
\usepackage{mathtools}
\usepackage{kotex}

\title[WS transform on kinetic models]{The Watanabe-Strogatz transform and constant of motion functionals for kinetic vector models}

\author[Park]{Hansol Park}
\address[Hansol Park]{\newline Department of Mathematical Sciences\newline Seoul National University, Seoul 08826, Republic of Korea}
\email{hansol960612@snu.ac.kr}

\newtheorem{theorem}{Theorem}[section]
\newtheorem{lemma}{Lemma}[section]
\newtheorem{corollary}{Corollary}[section]
\newtheorem{proposition}{Proposition}[section]
\newtheorem{remark}{Remark}[section]
\newtheorem{example}{Example}[section]

\newtheorem{definition}{Definition}[section]

\newcommand{\bbr}{\mathbb R}

\newcommand{\bbs}{\mathbb S}

\newcommand{\bbt} {\mathbb T}

\newcommand{\bbc}{\mathbb C}
\allowdisplaybreaks

\makeatletter
\@namedef{subjclassname@2020}{\textup{2020} Mathematics Subject Classification}
\makeatother

\begin{document}

\date{\today}

\subjclass[2020]{70G60, 34D06, 70F10} \keywords{Watanabe-Strogatz transform, cross ratio, push-forward solution, constant of motion}

\thanks{\textbf{Acknowledgment.} The work of H. Park is supported by NRF-2020R1A2C3A01003881(National Research Foundation of Korea).}

\begin{abstract}
We present a kinetic version of the Watanabe-Strogatz(WS) transform for vector models in this paper. From the generalized WS-transform, we obtain the cross-ratio type constant of motion functionals for kinetic vector models under suitable conditions. We present the sufficient and necessary conditions for the existence of the suggested constant of motion functional. As an application of the constant of motion functional, we provide the instability of bipolar states of the kinetic swarm sphere model. We also provide the WS-transform and constant of motion functional for non-identical kinetic vector models. 
\end{abstract}

\maketitle \centerline{\date}


\section{Introduction}
\setcounter{equation}{0}

Collective behaviors of synchronous systems are ubiquitous in nature. e.g., schooling of fish \cite{Aoki}, flashing of fireflies \cite{BB}, swarm of bacteria, flocks of birds, cooperative robot system \cite{DB, DF}, etc. From two pioneers Kuramoto \cite{Kura-1, Kura-2} and Winfree \cite{Winf-1, Winf-2}, studies on the synchronous system has been researched widely \cite{De, HK, Zhu}. One of the important concepts for the system is a constant of motion(conserved quantity). From a constant of motion, we can reduce the system. In \cite{HKPR, Lo, Lo-1, Wa-St-1, Wa-St-2}, constant of motion and transform of the Kuramoto model, the swarm sphere model, and matrix Riccati equations were introduced. So far, the WS-transform have been applied to ODE systems. The goal of this paper is to find a kinetic version of the WS-transform.\\

In this paper, we study the following kinetic system defined on the unit sphere $\bbs^d$ given as
\begin{align}\label{A-1}
\begin{cases}
\partial_t\rho_t+\nabla\cdot(\rho_t v_t)=0,\quad \rho_t\in\mathcal{P}(\bbs^d),\quad t>0,\\
\displaystyle v_t(x)=\Omega x+X(t)-\langle x, X(t)\rangle x,\quad \forall~x\in\bbs^d\subset\bbr^{d+1},
\end{cases}
\end{align}
where $\Omega$ is a skew-symmetric matrix with size $(d+1)\times (d+1)$ and $\mathcal{P}(\bbs^d)$ is a set of continuous probability density functions on $\bbs^d$. Since we denote $\rho_t(x):=\rho(t, x)$, we do not mentioned the initial data $\rho_0$. This system \eqref{A-1} is a kinetic version of the following ODE system defined on the unit sphere $\bbs^d$:
\begin{align}\label{A-2-0}
\begin{cases}
\displaystyle\frac{d}{dt}x_i=\Omega x_i+X-\langle x_i, X\rangle x_i,\quad t>0,\\
x_i(0)=x_i^0\in\bbs^d,\quad \forall~i\in\mathcal{N}:=\{1, 2, \cdots, N\}.
\end{cases}
\end{align}
System \eqref{A-2-0} is considered in \cite{Lo} as a generalized vector version of Kuramoto type models. This system is a general form of the swarm sphere model(a.k.a. Lohe sphere model) \cite{JC, Lo-3, Ta, TM}, the swarm sphere model with frustration \cite{HKPR} and the Winfree sphere model \cite{Park}(See Example \ref{E2.1} for details).\\

In this paper, we provide affirmative answers to the following analytical questions.\\
\begin{itemize}
\item (Q1): is there any Watanabe-Strogatz type transform for a kinetic system \eqref{A-1}?\\

\item (Q2): if so, is there any constant of motion functional for a kinetic system \eqref{A-1}?\\
\end{itemize}

The main results of this paper are four-fold. First, we find a kinetic version of the WS-transform on vector models. Using the push-forward of a measure, we could obtain a canonical form of the kinetic version of the WS-transform. For the kinetic system \eqref{A-1}, we define a time-dependent map $M_t: \bbs^d\to\bbs^d$ which satisfies
\begin{align}\label{A-2}
\begin{cases}
\displaystyle\frac{d}{dt}M_t(x)=\Omega x+X(t)-\langle x, X(t)\rangle x,\quad t>0,\vspace{0.2cm}\\
M_0(x)=x,\quad \forall ~x\in\bbs^d.
\end{cases}
\end{align}
Then, we know that $\rho_t=M_t\# \rho_0$ is a solution of system \eqref{A-1}(See Theorem \ref{Thm3.1}). This implies that if $X(t)$ is known for all $t\in[0, \infty)$ then a PDE system \eqref{A-1} can be reduced into an ODE system \eqref{A-2}. Since $X$ is just a time-dependent function, $X(t)$ $(t\geq0)$ can depend on previous solution data $\{\rho_s:s\in[0, t)\}$. In this case, we can not reduce the PDE system into the ODE system(See Section \ref{sec:3.1}).

Second, we use a kinetic version of the WS-transform to study constants of motion for a kinetic system \eqref{A-1}. Recall the particle version of system \eqref{A-1} is given as an ODE system \eqref{A-2-0} and its constants of motion are
\[
\frac{\|x_{i_1}-x_{i_2}\|^2\cdot \|x_{i_3}-x_{i_4}\|^2}{\|x_{i_2}-x_{i_3}\|^2\cdot\|x_{i_4}-x_{i_1}\|^2},\quad\prod_{\ell=1}^k\frac{\|x_{i_{2\ell-1}}-x_{i_{2\ell}}\|^2}{\|x_{i_{2\ell}}-x_{i_{2\ell+1}}\|^2},
\]
where $k\geq2$ and $i_{2k+1}=i_1$. From formal calculations, we can obtain the following candidates of constant of motion functionals(See Section \ref{sec:3.2.1}):
\begin{align}
\begin{aligned}\label{A-4}
&\mathcal{H}_p[\rho]:=\int_{(\bbs^d)^4}\left(\frac{\|x_1-x_2\|^2\cdot \|x_3-x_4\|^2}{\|x_2-x_3\|^2\cdot\|x_4-x_1\|^2}\right)^p\rho(x_1)\rho(x_2)\rho(x_3)\rho(x_4)d\sigma_{x_1}d\sigma_{x_2}d\sigma_{x_3}d\sigma_{x_4},\\
&\mathcal{H}_{p, k}[\rho]:= \int_{(\bbs^d)^{2k}}\left(\prod_{\ell=1}^k\frac{\|x_{2\ell-1}-x_{2\ell}\|^2}{\|x_{2\ell}-x_{2\ell+1}\|^2}\right)^p\rho(x_1)\rho(x_2)\cdots \rho(x_{2k})d\sigma_{x_1}d\sigma_{x_2}\cdots d\sigma_{x_{2k}}.
\end{aligned}
\end{align}
Here, $d\sigma_{x_i}$ is a surface measure on $\bbs^d$ with respect to a variable $x_i$, and in this paper we use $d\sigma_{x_i}=d\sigma_i$ if it does not make confusion. We show that functionals $\mathcal{H}_p$ and $\mathcal{H}_{p, k}$ are constants of motion for system \eqref{A-1} from the property of $M_t$ defined in \eqref{A-2}(See Section \ref{sec:3.2.2}). We also provide the sufficient and necessary conditions for existence of $\mathcal{H}_p$ and $\mathcal{H}_{p, k}$ as follows:
\[
p\in\left(-\frac{d}{2},\frac{d}{2}\right),
\]
where $d$ is the dimension of the domain $\bbs^d$.

Third, we use a constant of motion functional $\mathcal{H}_p$ to show the instability of bipolar states for the kinetic swarm sphere model \eqref{E-1}. Begin with the following observation:
\[
y_1\neq y_2\in\bbs^d,\quad 0<q<1\quad\Longrightarrow\quad\mathcal{H}_p[q\delta_{y_1}+(1-q)\delta_{y_2}]=\infty,
\]
and
\[
\rho_0\in\mathcal{P}(\bbs^d)\quad\Rightarrow\quad\mathcal{H}_p[\rho_0]<\infty,
\]
we guess that bipolar states can not be emerged in the kinetic swarm sphere model. We formulate this idea in Lemma \ref{Lem5.1} and Theorem \ref{Thm5.1}, and we obtain the complete synchronization of the kinetic swarm sphere model in the different view point studied in \cite{F-L}. On the other hand, we know that for any real number $\varepsilon>0$ there exists a continuous probability density function $\rho\in\mathcal{P}(\bbs^d)$ satisfies
\[
W_2(\rho, q\delta_{y_1}+(1-q)\delta_{y_2})<\varepsilon,
\]
where $W_2$ is the Wasserstein-2 metric. However, we have
\[
\mathcal{H}_p[\rho]<\infty,\quad \mathcal{H}_p[q\delta_{y_1}+(1-q)\delta_{y_2}]=\infty.
\]
Since $\mathcal{H}_p$ is a constant of motion functional, we can conclude that the state $q\delta_{y_1}+(1-q)\delta_{y_2}$ is unstable in some sense(See Section \ref{sec:5}).\\

Forth, we consider non-identical particle/kinetic vector models. Originally, the WS-transform have been only applied to identical systems \eqref{A-2-0}. i.e. all free-flow terms involved with $\Omega$ are identical. We apply the WS-transform to non-identical particle/kinetic sphere models, and we obtain their constant of motion functionals.(See section \ref{sec:6}).\\

The rest of this paper is organized as follows. In Section \ref{sec:2}, we introduce the previous result of the WS-transform and constant of motion functionals for the Kuramoto model and vector models. In Section \ref{sec:3}, we construct a kinetic version of the WS-transform and constant of motion functionals for vector models . In Section \ref{sec:4}, we provide the sufficient and necessary conditions for the existence of functionals $\mathcal{H}_p$ and $\mathcal{H}_{p, k}$. In Section \ref{sec:5}, we study the instability of bipolar states of the kinetic swarm sphere model as an application of constant of motion functionals. In Section \ref{sec:6}, we provide the WS-transform and constant of motion functionals for the non-identical systems. Finally, Section \ref{sec:7} is devoted to a summary of the paper.\\


\section{Preliminaries}\label{sec:2}
\setcounter{equation}{0}

In this section, we provide previous results on Watanabe-Strogatz(WS) transform and its generalization to the vector model defined on the unit sphere $\bbs^d$. Since the purpose of this paper is constructing a kinetic version of the WS-transform, now we briefly review the original WS-transform and its generalization to vector models.

\subsection{The WS-transform}\label{sec:2.1}
In this subsection, we introduce the original WS-transform defined on Kuramoto type models. This transform has been introduced in \cite{Wa-St-1, Wa-St-2} by Watanabe and Strogatz. The Watanabe-Strogatz(WS) transform reduce $N$ equations of the Kuramoto model into three real equations.  Let a Kuramoto type equation be given as the following form:
\begin{align}\label{B-(1)}
\begin{cases}
\displaystyle\dot{\theta}_i=\omega-\frac{\mathrm{i}}{2}\left(H(t)e^{-\mathrm{i}\theta_i}-\bar{H}(t)e^{\mathrm{i}\theta_i}\right),\quad t>0,\\
\theta_i(0)=\theta_i^0,\quad\forall~i\in\mathcal{N},
\end{cases}
\end{align}
where $\omega\in\bbr$ is a natural frequency of the system and a complex function $H$ depends on $t$. If we substitute $H(t)=\frac{\kappa}{N}\sum_{k=1}^N e^{\mathrm{i}\theta_k}$ into system \eqref{B-(1)}, then we can obtain the identical Kuramoto model:
\begin{align}\label{B-2-0}
\begin{cases}
\displaystyle\dot{\theta}_i=\omega+\frac{\kappa}{N}\sum_{k=1}^N\sin(\theta_k-\theta_i)\quad t>0,\\
\theta_i(0)=\theta_i^0,\quad\forall~i\in\mathcal{N}.
\end{cases}
\end{align}
So we know that system \eqref{B-(1)} is a general type of the Kuramoto model \eqref{B-2-0}. Now, we assume that $e^{\mathrm{i}\theta_j}$ can be expressed as the following form:
\begin{align}\label{B-(2)}
e^{\mathrm{i}\theta_j(t)}=\frac{z(t)+e^{\mathrm{i}\xi_j(t)}}{1+\bar{z}(t)e^{\mathrm{i}\xi_j(t)}},\quad \xi_j(t)=\xi_j^0+\alpha(t),
\end{align}
where time-dependent complex numbers $z$ and $\alpha$ are independent of index $j\in\mathcal{N}$. For constants $\{\xi_j^0\}_{j=1}^N$, we can express $\{\theta_j(t)\}_{j=1}^N$ as functions of $z(t)$ and $\alpha(t)$. i.e. there exists a correspondence 
\[
\{\theta_j(t)\}_{j=1}^N\quad \longleftrightarrow\quad (z(t), \alpha(t), \{\xi_j^0\}_{j=1}^N).
\]
If we impose the dynamics of $z$ and $\alpha$ as follows:
\[
\begin{cases}
\displaystyle\dot{z}=\mathrm{i}\omega z+\frac{H}{2}-z^2\frac{\bar{H}}{2},\\
\displaystyle\dot{\alpha}=\omega+\frac{\mathrm{i}}{2}(z\bar{H}-\bar{z}H),\quad t>0,
\end{cases}
\]
and $\{\theta_j\}_{j=1}^N$ defined as \eqref{B-(2)}, then  $\{\theta_j(t)\}_{j=1}^N$ is a solution of system \eqref{B-(1)}. Now, we define an operator $M_t:\bbc\to\bbc$ as follows:
\[
M_t(w)=\frac{z(t)+we^{\mathrm{i}\alpha(t)}}{1+\bar{z}(t)we^{\mathrm{i}\alpha(t)}}.
\]
By relation \eqref{B-(2)}, we have
\[
e^{\mathrm{i}\theta_j}=M_t(e^{\mathrm{i}\xi_j^0}).
\]
We define a cross ratio functional $\lambda$ as follows:
\[
\lambda(\alpha_1,\alpha_2,\alpha_3,\alpha_4)=\frac{(\alpha_1-\alpha_2)(\alpha_3-\alpha_4)}{(\alpha_2-\alpha_3)(\alpha_4-\alpha_1)},
\]
where $\alpha_1, \alpha_2, \alpha_3, \alpha_4$ are unit complex numbers. i.e. $\alpha_l\in\bbc$ and $|\alpha_l|=1$ for $l\in\{1, 2, 3, 4\}$. From simple calculations, we know that a transformation $M_t$ preserves a functional $\lambda$. i.e. for any $t\geq0$, we have
\[
\lambda(\alpha_1,\alpha_2,\alpha_3,\alpha_4)=\lambda(M_t(\alpha_1), M_t(\alpha_2), M_t(\alpha_3), M_t(\alpha_4)).
\]
This implies that if $i, j, k, l$ are distinct indices, then 
\[
\lambda(e^{\mathrm{i}\theta_i(t)}, e^{\mathrm{i}\theta_j(t)}, e^{\mathrm{i}\theta_k(t)}, e^{\mathrm{i}\theta_l(t)})=\lambda(e^{\mathrm{i}\xi_i^0}, e^{\mathrm{i}\xi_j
^0}, e^{\mathrm{i}\xi_k^0}, e^{\mathrm{i}\xi_l^0})=\lambda(e^{\mathrm{i}\theta_i^0}, e^{\mathrm{i}\theta_j^0}, e^{\mathrm{i}\theta_k^0}, e^{\mathrm{i}\theta_l^0})
\]
for all $t\geq0$. If we define the following functional:
\[
\lambda_{ijkl}(\Theta)=\lambda(e^{\mathrm{i}\theta_i}, e^{\mathrm{i}\theta_j}, e^{\mathrm{i}\theta_k}, e^{\mathrm{i}\theta_l})
\]
where $\Theta=\{\theta_j\}_{j=1}^N$, then we have
\[
\lambda_{ijkl}(\Theta(t))=\lambda_{ijkl}(\Theta^0)
\]
and this means $\lambda_{ijkl}$ is a constant of motion functional for system \eqref{B-(1)}.

\subsection{A Higher-dimensional WS-transform}\label{sec:2.2}
In this subsection, we introduce a generalization of the WS-transform on vector models. This generalized transform was introduced in \cite{Lo} by Lohe. Let a particle model be given as follows:
\begin{align}\label{B-(3)}
\begin{cases}
\dot{x}_i=\Omega x_i+X-\langle x_i, X\rangle x_i,\quad t>0,\\
x_i(0)=x_i^0,\quad \forall~i\in\mathcal{N}.
\end{cases}
\end{align}
where $X(t)\in \bbr^{d+1}$ for all $t\geq0$. We have the following simple lemma on the collisionless of system \eqref{B-(3)}.

\begin{lemma}\label{Lem2.1}
Suppose that the initial data $\{x_j^0\}_{j=1}^N$ satisfy $x_k^0\neq x_l^0$ for some $k, l\in\mathcal{N}$ and let $\mathcal{X}=\{x_j\}_{j=1}^N$ be a solution of system \eqref{B-(3)}. Then we have
\[
x_k(t)\neq x_l(t)\quad\forall t\geq0.
\]
\end{lemma}
\begin{proof}
From direct calculations, we have
\begin{align*}
\frac{1}{2}\frac{d}{dt}\|x_k-x_l\|^2&=\langle \dot{x}_k-\dot{x}_l, x_k-x_l\rangle
=\big\langle\langle X, x_l\rangle x_l-\langle X, x_k\rangle x_k, x_k-x_l\big\rangle\\
&=-(1-\langle x_k, x_l\rangle)\cdot \langle X, x_k+x_l\rangle=-\frac{1}{2}\|x_k-x_l\|^2\cdot\langle X, x_k+x_l\rangle.
\end{align*}
This yields,
\[
\|x_k(t)-x_l(t)\|=\|x_k^0-x_l^0\|\exp\left(-\frac{1}{2}\int_0^\top \langle X(s), x_l(s)+x_k(s)\rangle ds\right).
\]
By the assumption $x_k^0\neq x_l^0$, we have $x_k(t)\neq x_l(t)$ for all $t\geq0$.
\end{proof}

Throughout this paper, we will consider that the initial data $\{x_j^0\}$ satisfy $x_k^0\neq x_l^0$ for all disjoint indices $k\neq l\in\mathcal{N}$. From Lemma \ref{Lem2.1}, automatically we have $x_k(t)\neq x_l(t)$ for all $k\neq l\in\mathcal{N}$ and $t\geq0$. Now, we introduce some examples of system \eqref{B-(3)}.

\begin{example}\label{E2.1}
We can substitute various $X$ into system \eqref{B-(3)}.

\noindent(1) $X=\frac{\kappa}{N}\sum_{k=1}^Nx_k$. We get the swarm sphere model as follows:
\begin{align}\label{B-(3)-1}
\begin{cases}
\displaystyle\dot{x}_i=\Omega x_i+\frac{\kappa}{N}\sum_{k=1}^N(x_k-\langle x_i, x_k\rangle x_i),\quad t>0,\\
x_i(0)=x_i^0\in\bbs^d,\quad \forall~i\in\mathcal{N}.
\end{cases}
\end{align}
\vspace{0.5cm}
\noindent(2) $X=\frac{\kappa}{N}\sum_{k=1}^NVx_k$. We get the swarm sphere model with a frustration matrix $V\in\bbr^{(d+1)\times(d+1)}$ as follows:
\[
\begin{cases}
\displaystyle\dot{x}_i=\Omega x_i+\frac{\kappa}{N}\sum_{k=1}^N(Vx_k-\langle x_i, Vx_k\rangle x_i),\quad t>0,\\
x_i(0)=x_i^0\in\bbs^d,\quad \forall~i\in\mathcal{N}.
\end{cases}
\]
This system was studied in \cite{HKPR}.\vspace{0.5cm}

\noindent(3) $X=\frac{\kappa}{N}\sum_{k=1}^NI(x_k)e$. We get the identical Winfree-sphere model as follows:
\[
\begin{cases}
\displaystyle\dot{x}_i=\Omega x_i+\frac{\kappa}{N}(e-\langle x_i, e\rangle x_i)\sum_{k=1}^NI(x_k),\quad t>0,\\
x_i(0)=x_i^0\in\bbs^d,\quad \forall~i\in\mathcal{N},
\end{cases}
\]
where $I$ is the influence function of the system and $e=[\underbrace{0, \cdots, 0}_{\text{$d$-times}}, 1]^\top$ is the north-pole of the unit sphere $\bbs^d$ embedded in $\bbr^{d+1}$. This system was studied in \cite{Park}.\vspace{0.5cm}

\noindent(4) $X(t)=\frac{\kappa}{N}\sum_{k=1}^Nx_k(t-\tau)$. We get the swarm sphere model with a time-delay $\tau>0$ as follows
\begin{align*}
\begin{cases}
\displaystyle\dot{x}_i(t)=\Omega x_i(t)+\frac{\kappa}{N}\sum_{k=1}^N(x_k(t-\tau)-\langle x_i(t), x_k(t-\tau)\rangle x_i(t)),\quad t>0,\\
x_i(s)=\varphi_i^0(s)\in\bbs^d,\quad \forall~-\tau\leq s\leq 0,\quad i\in\mathcal{N}.
\end{cases}
\end{align*}
This system was studied in \cite{DHK}.
\end{example}

Now, we consider the following generalization of the WS-transform:
\begin{align}\label{B-3-2}
x_i=w+\frac{(u_i+w)(1-\|w\|^2)}{\|u_i+w\|^2},\quad u_i=-w+\frac{(x_i-w)(1-\|w\|^2)}{\|x_i-w\|^2}.
\end{align}
If we define $W_w:\bbs^{d}\to\bbs^{d}$ as follows:
\[
W_w(u)=v+\frac{(u+w)(1-\|w\|^2)}{\|u+w\|^2},
\]
then we can easily check that
\[
W_0=\mathrm{id}_{\bbs^d}\quad\text{and}\quad W_{-w}=W_w^{-1}\quad\forall~w\in\bbs^d.
\]
Using $W_w$, we can rewrite relation \eqref{B-3-2} as
\[
x_i=W_w(u_i),\quad u_i=W_w^{-1}(x_i).
\]
Now we assume that $(w(t), R(t))\in\bbs^d\times \mathrm{SO}(d+1)$ is a solution of the following system:
\begin{align}\label{B-(4)}
\begin{cases}
\displaystyle \dot{w}=\Omega w+\frac{1}{2}(1+\|w\|^2)X-\langle w, X\rangle w,\quad w(0)=0,\\
\dot{R}=\Omega R+Xw^\top R-wX^\top R,\quad 
R(0)=I_{d+1}.
\end{cases}
\end{align}
Then we know that
\[
x_i(t)=W_{w(t)}(R(t)x_i^0)
\]
is a solution of system \eqref{B-(3)}. If we define
\[
M_t(x)=W_{w(t)}(R(t)x),
\]
then this yields
\begin{align*}
M_t(x)=w(t)+\frac{(R(t)x+w(t))(1-\|w(t)\|^2)}{\|R(t)x+w(t)\|^2}.
\end{align*}
By the given initial condition, $w(0)=0$ and $R(0)=I_{d+1}$, implies
\[
M_0(x)=0+\frac{(x+0)(1-0^2)}{\|x+0\|^2}=x,\quad \forall x\in\bbs^{d},
\]
and this yields $M_0=\mathrm{Id}_{\bbs^d}$. Then the solution $\{x_i\}_{i=1}^N$ of system \eqref{B-(3)} can be expressed as
\[
x_i(t)=M_t(x_i^0)\quad t\geq0,\quad i\in\mathcal{N}.
\]
We know that if we define the cross ratio of four distinct vectors $u_1, u_2, u_3, u_4\in\bbs^d$ by 
\[
\mathcal{C}(u_1, u_2, u_3, u_4):=\frac{(1-u_1\cdot u_2)(1-u_3\cdot u_4)}{(1-u_2\cdot u_3)(1-u_4\cdot u_1)}=\frac{\|u_1-u_2\|^2\cdot \|u_3-u_4\|^2}{\|u_2-u_3\|^2\cdot\|u_4-u_1\|^2},
\]
then $M_t$ leaves $C$ invariant. i.e.
\[
\mathcal{C}(u_1, u_2, u_3, u_4)=\mathcal{C}(M_t(u_1), M_t(u_2), M_t(u_3), M_t(u_4))\quad\forall ~t\geq0.
\]
So we know that 
\[
\mathcal{C}_{ijkl}(\mathcal{X}(t)):=\mathcal{C}(x_i(t), x_j(t), x_k(t), x_l(t))=\mathcal{C}(x_i^0, x_j^0, x_k^0, x_l^0)=\mathcal{C}_{ijkl}(\mathcal{X}^0)\quad \forall~t\geq0
\]
is a constant of motion for system \eqref{B-(3)}, where $\mathcal{X}:=\{x_j\}_{j=1}^N$. For the well-definedness of the cross ratio $\mathcal{C}$, we define the following set of cycles with length $2k$ for $2\leq k\in\mathbb{N}$:
\[
C_k :=\{(i_1, i_2, \cdots, i_{2k}): i_1\neq i_{2k}~\text{and}~ i_n\neq i_{n+1}~\forall n=1, 2, \cdots, 2k-1\}.
\]
Based on the above argument, we get the following theorem.

\begin{theorem}[\cite{HKPR, Lo}]\label{Thm2.1}
Let $\mathcal{X}=\{x_j\}_{j=1}^N$ be a solution of system \eqref{B-(3)}.  For any $(i_1, i_2, i_3, i_4)\in C_2$, $\mathcal{C}_{i_1i_2i_3i_4}$ is a constant of motion functional for system \eqref{B-(3)}.
\end{theorem}

Now, we introduce induced constant of motion functionals from Theorem \eqref{Thm2.1}. For any $\Lambda=(i_1, i_2, \cdots, i_{2k})\in C_k$, we define $\mathcal{C}_\Lambda$ as follows:
\[
\mathcal{C}_\Lambda(\mathcal{X}):=\prod_{\ell=1}^k\frac{\|x_{i_{2\ell-1}}-x_{i_{2\ell}}\|^2}{\|x_{i_{2\ell}}-x_{i_{2\ell+1}}\|^2},
\]
where we denote $i_{2k+1}=i_1$. We can prove that $\mathcal{C}_\Lambda$ is also a constant of motion.
\begin{corollary}\label{Cor2.1}
Let $\mathcal{X}=\{x_j\}$ be a solution of \eqref{B-(3)}.  For any $\Lambda=(i_1, i_2, \cdots, i_{2k})\in C_k$,  $\mathcal{C}_\Lambda$ is a constant of motion. i.e.
\[
\frac{d}{dt}\mathcal{C}_\Lambda(\mathcal{X}(t))=0.
\]
\end{corollary}
\begin{proof}
Here, we set $i_1=i_{2k+1}$ and $i_2=i_{2k+2}$. By simple calculations, we have
\begin{align*}
\mathcal{C}_\Lambda(\mathcal{X}):&=\prod_{\ell=1}^k\frac{\|x_{i_{2\ell-1}}-x_{i_{2\ell}}\|^2}{\|x_{i_{2\ell}}-x_{i_{2\ell+1}}\|^2}\\
&=\prod_{\ell=1}^k\frac{\|x_{i_1}-x_{i_{2\ell}}\|^2\cdot\|x_{i_{2\ell+1}}-x_{i_{2\ell+2}}\|^2}{\|x_{i_{2\ell}}-x_{i_{2\ell+1}}\|^2\cdot\|x_{i_{2\ell+2}}-x_{i_1}\|^2}
=\prod_{\ell=1}^k\mathcal{C}_{i_1i_{2\ell}i_{2\ell+1}i_{2\ell+2}}.
\end{align*}
From Theorem \ref{Thm2.1}, we know that each $\mathcal{C}_{i_1i_{2\ell}i_{2\ell+1}i_{2\ell+2}}$ is a constant of motion for system \eqref{B-(3)} for all $1\leq\ell\leq k$. Since $\mathcal{C}_\Lambda(\mathcal{X})$ is a product of constants of motion, we know that it also a constant of motion for system \eqref{B-(3)}.
\end{proof}
So far, we introduced the WS-transform of vector models on the unit sphere($\bbs^d$). The WS-transform also can be applied to the double sphere model on $\bbs^{d_1}\times \bbs^{d_2}$(for details see \cite{Lo-2}).

\section{Kinetic version of the higher-dimensional WS-transform}\label{sec:3}
\setcounter{equation}{0}

In this section, we study a kinetic version of the WS-transform on vector models. We express a kinetic version of the WS-transform as a push-forward of the initial data $\rho_0\in\mathcal{P}(\bbs^d)$.

\subsection{Kinetic version of the WS-transform on vector models}\label{sec:3.1}
In this subsection, we introduce a kinetic version of system \eqref{B-(3)} and the generalized WS-transform. A kinetic version of system \eqref{B-(3)} can be written as
\begin{align}\label{C-(1)}
\begin{cases}
\partial_t\rho_t+\nabla\cdot(\rho_t v_t)=0,\quad t>0,\\
v_t(x)=\Omega x+X(t)-\langle x, X(t)\rangle x,\quad \forall~x\in\bbs^d,
\end{cases}
\end{align}
where $\rho_t$ is a probability density function on the unit sphere $\bbs^d$(i.e. $\rho_t\in\mathcal{P}(\bbs^d)$) and $\nabla$ is a covariant derivative on the unit sphere $\bbs^d$.  If we substitute
\begin{align}\label{C-(2)}
X(t)=\int y\rho_t(y)d\sigma_y,
\end{align}
into \eqref{C-(1)}, then we have the following kinetic swarm sphere model:
\begin{align}\label{C-(3)}
\begin{cases}
\partial_t\rho_t+\nabla\cdot(\rho_t v_t)=0\quad t>0,\\
\displaystyle v_t(x)=\Omega x+\int_{\bbs^d}(y-\langle x, y\rangle x)\rho_t(y)d\sigma_y\quad\forall~x\in\bbs^d.
\end{cases}
\end{align}
A special case of system \eqref{C-(1)} is the kinetic swarm sphere model \eqref{C-(3)} introduced in \cite{F-L, HKLN}.

Now we try to express a solution of system \eqref{C-(1)} as a push-forward of the initial data $\rho_0$. Let $w$ be a solution of the following system:
\begin{align}\label{C-3-1-0}
\begin{cases}
\displaystyle\dot{w}=\Omega w+\frac{1}{2}(1+\|w\|^2)X-\langle w, X\rangle w,\quad t>0,\\
w(0)=0,
\end{cases}
\end{align}
and we define a time-dependent map $M_t:\bbs^d\to\bbs^d$ as follows:
\begin{align}\label{C-3-1}
M_t(x)=w(t)+\frac{(u(t)+w(t))(1-\|w(t)\|^2)}{\|u(t)+w(t)\|^2}.
\end{align}
We want to construct the dynamics of $u:[0, \infty)\to\bbr^{d+1}$ which satisfies
\begin{align*}
\begin{cases}
\displaystyle\frac{d}{dt}M_t(x)=\Omega M_t(x)+X-\langle M_t(x), X\rangle M_t(x),\quad t>0,\\
M_0(x)=x,\quad \forall~x\in\bbs^d.
\end{cases}
\end{align*}
Since $M_t(x)=x$ and $w(0)=0$, we easily get
\begin{align}\label{C-3-1-1}
u(0)=x.
\end{align}
From \eqref{C-3-1}, we have
\begin{align*}
u(t)=-w(t)+\frac{(M_t(x)-w(t))(1-\|w(t)\|^2)}{\|M_t(x)-w(t)\|^2}.
\end{align*}
Then the temporal derivative of $u$ is
\begin{align}
\begin{aligned}\label{C-3-3}
\dot{u}&=-\dot{w}+\frac{(\frac{d}{dt}M_t(x)-\dot{w})(1-\|w\|^2)}{\|M_t(x)-w\|^2}-\frac{2\langle w,\dot{w}\rangle(M_t(x)-w)}{\|M_t(x)-w\|^2}\\
&\hspace{0.5cm}-2\left\langle \frac{d}{dt}M_t(x)-\dot{w}, M_t(x)-w \right\rangle\frac{(M_t(x)-w)(1-\|w\|^2)}{\|M_t(x)-w\|^4}.
\end{aligned}
\end{align}
Now, we introduce a lemma to simplify the calculation \eqref{C-3-3}.
\begin{lemma}\label{L3.1}
Let $M_t(x)$, $w$, and $u$ be functions defined above. Then we have the following calculations:
\begin{align*}
&(1)~\langle w, \dot{w}\rangle=\frac{1}{2}(1-\|w\|^2)\langle w, X\rangle,\\
&(2)~\frac{d}{dt}M_t(x)-\dot{w}=\Omega(M_t(x)-w)+\frac{1}{2}(1-\|w\|^2)X-\langle M_t(x), X\rangle M_t(x)+\langle w, X\rangle w,\\
&(3)~\left\langle \frac{d}{dt}M_t(x)-\dot{w}, M_t(x)-w \right\rangle=-\frac{1}{2}\|M_t(x)-w\|^2\langle M_t(x)+w, X\rangle.
\end{align*}
\end{lemma}

\begin{proof}
(1) By direct calculations, we have
\begin{align*}
\langle w, \dot{w}\rangle&=\left\langle w,  \Omega w+\frac{1}{2}(1+\|w\|^2)X-\langle w, X\rangle w\right\rangle=\frac{1}{2}(1+\|w\|^2)\langle w, X\rangle-\langle w, X\rangle \|w\|^2\\
&=\frac{1}{2}(1-\|w\|^2)\langle w, X\rangle.
\end{align*}

(2) By simple arrangement, we get the desired result.\\

(3) Using the result of (2), we get
\begin{align*}
&\left\langle \frac{d}{dt}M_t(x)-\dot{w}, M_t(x)-w \right\rangle\\
=&\left\langle\Omega(M_t(x)-w)+\frac{1}{2}(1-\|w\|^2)X-\langle M_t(x), X\rangle M_t(x)+\langle w, X\rangle w, M_t(x)-w \right\rangle\\
=&\frac{1}{2}(1-\|w\|^2)\langle X, M_t(x)-w\rangle-\langle M_t(x), X\rangle\langle M_t(x), M_t(x)-w\rangle+\langle w, X\rangle\langle w, M_t(x)-w\rangle\\
=&\langle M_t(x), X\rangle\left(\frac{1}{2}(1-\|w\|^2)-1+\langle M_t(x), w\rangle\right)-\langle w, X\rangle\left(\frac{1}{2}(1-\|w\|^2)-\langle w, M_t(x)\rangle+\|w\|^2\right)\\
=&-\frac{1}{2}\langle M_t(x), X\rangle\|M_t(x)-w\|^2-\frac{1}{2}\langle w, X\rangle\|M_t(x)-w\|^2\\
=&-\frac{1}{2}\|M_t(x)-w\|^2\langle M_t(x)+w, X\rangle.
\end{align*}
\end{proof}

Now we substitute the result of Lemma \ref{L3.1} into \eqref{C-3-3} to get
\begin{align}
\begin{aligned}\label{C-3-4}
\dot{u}&=-\dot{w}+\frac{1-\|w\|^2}{\|M_t(x)-w\|^2}\left(\Omega(M_t(x)-w)+\frac{1}{2}(1-\|w\|^2)X-\langle M_t(x), X\rangle M_t(x)+\langle w, X\rangle w\right)\\
&\hspace{0.5cm}-\frac{(1-\|w\|^2)\langle w, X\rangle}{\|M_t(x)-w\|^2}(M_t(x)-w)+\frac{(1-\|w\|^2)\langle M_t(x)+w, X\rangle}{\|M_t(x)-w\|^2}(M_t(x)-w)\\
&=-\dot{w}+\frac{(1-\|w\|^2)\langle M_t(x)+w, X\rangle}{\|M_t(x)-w\|^2}(M_t(x)-w)\\
&\hspace{0.5cm}+\frac{1-\|w\|^2}{\|M_t(x)-w\|^2}\Big(\Omega(M_t(x)-w)+\frac{1}{2}(1-\|w\|^2)X-\langle M_t(x), X\rangle M_t(x)\\
&\hspace{8cm}+\langle w, X\rangle w-\langle w, X\rangle(M_t(x)-w)\Big)\\
&=-\dot{w}+\langle M_t(x)+w, X\rangle (u+w)+\Omega(u+w)\\
&\hspace{0.5cm}+\frac{1-\|w\|^2}{\|M_t(x)-w\|^2}\left(\frac{1}{2}(1-\|w\|^2)X-\langle M_t(x)+w, X\rangle M_t(x)+2\langle w, X\rangle w\right)\\
&=-\Omega w-\frac{1}{2}(1+\|w\|^2)X+\langle w, X\rangle w+\langle M_t(x)+w, X\rangle (u+w)+\Omega(u+w)\\
&\hspace{0.5cm}+\frac{1}{2}\|u+w\|^2X-\frac{1-\|w\|^2}{\|M_t(x)-w\|^2}\langle M_t(x)+w, X\rangle M_t(x)+\frac{2(1-\|w\|^2)\langle w, X\rangle }{\|M_t(x)-w\|^2}w\\
&=\Omega u+\langle u, w\rangle X+\langle w, X\rangle w+\langle M_t(x)+w, X\rangle (u+w)-\frac{1-\|w\|^2}{\|M_t(x)-w\|^2}\langle M_t(x)+w, X\rangle M_t(x)\\
&\hspace{0.5cm}+\frac{2(1-\|w\|^2)\langle w, X\rangle }{\|M_t(x)-w\|^2}w.
\end{aligned}
\end{align}
By direct calculation, we have
\begin{align}
\begin{aligned}\label{C-3-4-1}
&\langle w, X\rangle w+\langle M_t(x)+w, X\rangle (u+w)-\frac{1-\|w\|^2}{\|M_t(x)-w\|^2}\langle M_t(x)+w, X\rangle M_t(x)
+\frac{2(1-\|w\|^2)\langle w, X\rangle }{\|M_t(x)-w\|^2}w\\
=&\left(1+\frac{2(1-\|w\|^2)}{\|M_t(x)-w\|^2}\right)\langle w, X\rangle w+\langle M_t(x)+w, X\rangle\left(u+w-\frac{1-\|w\|^2}{\|M_t(x)-w\|^2}M_t(x)\right)\\
=&\left(1+\frac{2(1-\|w\|^2)}{\|M_t(x)-w\|^2}\right)\langle w, X\rangle w-\langle M_t(x)+w, X\rangle\frac{1-\|w\|^2}{\|M_t(x)-w\|^2}w\\
=&\left\langle w+\frac{2(1-\|w\|^2)}{\|M_t(x)-w\|^2}w-\frac{1-\|w\|^2}{\|M_t(x)-w\|^2}(M_t(x)+w),X\right\rangle w\\
=&\left\langle w-\frac{1-\|w\|^2}{\|M_t(x)-w\|^2}(M_t(x)-w),X\right\rangle w\\
=&\left\langle w-(u+w),X\right\rangle w=-\langle u, X\rangle w.
\end{aligned}
\end{align}
Now we substitute \eqref{C-3-4-1} into \eqref{C-3-4} to obtain
\[
\dot{u}=\Omega u+\langle u, w\rangle X-\langle u, X\rangle w
\]
and this can be rewritten as
\begin{align}\label{C-3-4-2}
\dot{u}=(\Omega+Xw^\top - wX^\top )u. 
\end{align}
By \eqref{C-3-1-1}, we can substitute $u(t)=R(t)x$, where $R(t)$ is an orthogonal matrix with size $(d+1)\times (d+1)$ for all $t\geq0$ and especially $R(0)=I_{d+1}$. If we substitute $u(t)=R(t)x$ into \eqref{C-3-4-2}, we have 
\begin{align*}
\dot{R}x=(\Omega R+Xw^\top-wX^\top)R x.
\end{align*}
If we set $R$ is a solution of the following ODE system:
\begin{align}\label{C-3-6}
\begin{cases}
\dot{R}=(\Omega+Xw^\top-wX^\top)R,\quad t>0,\\
R(0)=I_{d+1},
\end{cases}
\end{align}
then relation \eqref{C-3-4-2} holds. Now we define the map $M_t:\bbs^d\to\bbs^d$ as follows:
\begin{align}\label{C-3-7}
M_t(x)=w(t)+\frac{(R(t)x+w(t))(1-\|w(t)\|^2)}{\|R(t)x+w(t)\|^2},
\end{align}
where $w$ is a solution of system \eqref{C-3-1-0} and $R$ is a solution of system \eqref{C-3-6}. Then we have
\[
\begin{cases}
\displaystyle\frac{d}{dt}M_t(x)=\Omega x+X(t)-\langle x, X(t)\rangle x,\quad t>0,\\
M_0(x)=x,\quad \forall~x\in\bbs^d,
\end{cases}
\]
since we constructed $R$ and $w$ to satisfy it. So, we can conclude that 
\[
\rho_t=M_t\# \rho_0
\]
is a solution of system \eqref{C-(1)} with the initial data $\rho_0$, since
\[
\partial_t(M_t\# \rho_t)=-\nabla \cdot(\rho_t v_t),
\]
where $v_t(M_t(x))=\frac{d}{dt}M_t(x)$. This implies that, if an explicit form of $X(t)$ is given for all $t\geq0$, then we can solve a PDE \eqref{C-(1)} from solving a ODE system \eqref{B-(4)}. We can summarize this result as the following theorem.
\begin{theorem}\label{Thm3.1}
Let $\rho_t$ be a solution of system \eqref{C-(1)} with the initial data $\rho_0$. Then $\rho_t=M_t\#\rho_0$, where $M_t$ is defined in \eqref{C-3-7}.
\end{theorem}

\begin{remark}
For the kinetic swarm sphere model \eqref{C-(3)}, $X(t)$ is given with \eqref{C-(2)}. Since $X(t)$ depends on $\rho_t$, we do not have an explicit form of $X(t)$ before solving the system. 
\end{remark}

\subsection{Constant of motion for the kinetic sphere model}\label{sec:3.2}
In this subsection, we study a constant of motion functionals for system \eqref{C-(1)}. Before we present a result, we provide an idea to get candidates of  constant of motion functionals for a kinetic model \eqref{C-(1)}.

\subsubsection{Formal derivation of constant of motion functionals}\label{sec:3.2.1}
Recall that constant of motion functionals for system \eqref{B-(3)} are given as
\[
\mathcal{C}_{i_1i_2i_3i_4}(\mathcal{X})=\frac{\|x_{i_1}-x_{i_2}\|^2\cdot \|x_{i_3}-x_{i_4}\|^2}{\|x_{i_2}-x_{i_3}\|^2\cdot\|x_{i_4}-x_{i_1}\|^2}\quad\text{and}\quad
\mathcal{C}_\Lambda(\mathcal{X}):=\prod_{\ell=1}^k\frac{\|x_{i_{2\ell-1}}-x_{i_{2\ell}}\|^2}{\|x_{i_{2\ell}}-x_{i_{2\ell+1}}\|^2},
\]
where $(i_1, i_2, i_3, i_4)\in C_2$ and $\Lambda=(i_1, i_2, \cdots, i_{2k})\in C_k$ for $k\geq2$. Using this fact, we can construct a cross-ratio type constant of motion functionals formally.\\

 First, we consider a cross-ratio type functional $\mathcal{C}_{i_1i_2i_3i_4}$, where $(i_1, i_2, i_3, i_4)\in C_2$. Since it is a constant of motion, we know that $\mathcal{C}_{i_1i_2i_3i_4}^p$ is also a constant motion for any $p\in\bbr$. Again, this implies that
\[
\frac{1}{|C_2|}\left(\sum_{ (i_1, i_2, i_3, i_4)\in C_2}\mathcal{C}_{i_1i_2i_3i_4}^p\right)
\]
is a constant of motion. Letting $N\to\infty$ yields the following formal calculation:
\begin{align}\label{C-(4)}
\frac{1}{|C_2|}\left(\sum_{ (i_1, i_2, i_3, i_4)\in C_2}\mathcal{C}_{i_1i_2i_3i_4}^p\right)\simeq\int_{(\bbs^d)^4}\mathcal{C}(x_1, x_2, x_3,x_4)^p\rho_t(x_1)\rho_t(x_2)\rho_t(x_3)\rho_t(x_4)d\sigma_1d\sigma_2d\sigma_3d\sigma_4,
\end{align}
where $\rho_t$ is a particle density function defined on the unit sphere $\bbs^d$ and $\mathcal{C}:(\bbs^d)^4\to \bbr$ defined as
\[
\mathcal{C}(x_1, x_2, x_3, x_4)=\frac{\|x_1-x_2\|^2\cdot\|x_3-x_4\|^2}{\|x_1-x_3\|^2\cdot\|x_2-x_4\|^2}.
\]
So we can expect that the R.H.S. of \eqref{C-(4)} is a candidate for constant of motion functionals.\\

Second, we consider an extended cross-ratio type functional $\mathcal{C}_\Lambda$, where $\Lambda\in C_k$ with $k\geq2$. Similarly we know that $\mathcal{C}_\Lambda^p$ is also a constant motion for $p\in\bbr$. This also implies that
\[
\frac{1}{|C_k|}\sum_{ \Lambda \in C_k}\mathcal{C}_\Lambda^p
\]
is a constant of motion. If we define
\[
\mathcal{C}_k(x_1, x_2, \cdots, x_{2k}):=\prod_{\ell=1}^k\frac{\|x_{2\ell-1}-x_{2\ell}\|^2}{\|x_{2\ell}-x_{2\ell+1}\|^2},
\]
then we can also obtain the following formal calculation
\begin{align}\label{C-(5)}
\frac{1}{|C_k|}\sum_{ \Lambda \in C_k}\mathcal{C}_\Lambda^p\simeq \int_{(\bbs^d)^{2k}}\mathcal{C}_k(x_1, x_2, \cdots, x_{2k})^p\rho_t(x_1)\rho_t(x_2)\cdots\rho_t(x_{2k})d\sigma_1d\sigma_2\cdots d\sigma_{2k},
\end{align}
and we can expect that the R.H.S. of \eqref{C-(5)} is also a candidate for constant of motion functionals.\\

Now we have two candidates of constant of motion functionals for system \eqref{C-(1)}, and we define these candidates functional as follows.
\begin{definition}\label{Def3.1}
We define functionals $\mathcal{H}_p$ and $\mathcal{H}_{p, k}$ for all $p\in\bbr$ and $k\geq2$ as follows:
\[
\mathcal{H}_p[\rho]:=\int_{(\bbs^d)^4}\mathcal{C}(x_1, x_2, x_3, x_4)^p\rho(x_1)\rho(x_2)\rho(x_3)\rho(x_4)d\sigma_1d\sigma_2d\sigma_3d\sigma_4
\]
and
\[
\mathcal{H}_{p, k}[\rho]:=\int_{(\bbs^d)^{2k}}\mathcal{C}_k(x_1, x_2, \cdots, x_{2k})^p \rho(x_1)\rho(x_2)\cdots\rho(x_{2k})d\sigma_1d\sigma_2\cdots d\sigma_{2k},
\]
where $\rho\in\mathcal{P}(\bbs^d)$.
\end{definition}

\begin{proposition}\label{Pro3.1}
We have the following properties of functionals $\mathcal{H}_p$ and $\mathcal{H}_{p, k}$ for $p\in\bbr$ and an integer $k\geq2$.
\begin{enumerate}
\item $\mathcal{H}_p$ is a special case of $\mathcal{H}_{p, k}$. i.e. $\mathcal{H}_{p, 2}=\mathcal{H}_p$.
\item $\mathcal{H}_{-p, k}=\mathcal{H}_{p, k}$.
\item If we put $p=0$, we have $\mathcal{H}_{0, k}\equiv 1$.
\end{enumerate}
\end{proposition}

\begin{proof}
\noindent (1) By the definition of $\mathcal{H}_p$ and $\mathcal{H}_{p, k}$, it is trivial.\\

\noindent (2) From direct calculations, we have
\begin{align*}
\mathcal{H}_{-p, k}[\rho]&=\int_{(\bbs^d)^{2k}}\mathcal{C}_k(x_1, x_2, \cdots, x_{2k})^{-p} \rho(x_1)\rho(x_2)\cdots\rho(x_{2k})d\sigma_1d\sigma_2\cdots d\sigma_{2k}\\
&=\int_{(\bbs^d)^{2k}}\mathcal{C}_k(x_2, x_3, \cdots, x_{2k}, x_1)^{p} \rho(x_1)\rho(x_2)\cdots\rho(x_{2k})d\sigma_1d\sigma_2\cdots d\sigma_{2k}\\
&=\int_{(\bbs^d)^{2k}}\mathcal{C}_k(x_1, x_2, \cdots, x_{2k-1}, x_{2k})^{p} \rho(x_{2k})\rho(x_1)\cdots\rho(x_{2k-1})d\sigma_1d\sigma_2\cdots d\sigma_{2k}\\
&=\mathcal{H}_{p, k}[\rho].
\end{align*}
Here, we used the definition of $\mathcal{C}_k$ in the second equality and variable change in the third equality. Since above relation holds for all $\rho\in\mathcal{P}(\bbs^d)$, we know that $\mathcal{H}_{-p, k}=\mathcal{H}_{p, k}$.\\

\noindent (3) Since $\mathcal{C}_k^0\equiv1$, we have
\[
\mathcal{H}_{0, k}[\rho]=\int_{(\bbs^d)^{2k}}\rho(x_1)\rho(x_2)\cdots\rho(x_{2k})d\sigma_1d\sigma_2\cdots d\sigma_{2k}=1.
\]
\end{proof}

Functionals $\mathcal{H}_p$ and $\mathcal{H}_{p, k}$ depend on $p\in\bbr$. We provide the importance of a choice of $p\in\bbr$ from an example. We consider system \eqref{C-(1)} on the circle($\bbs^1$). Then $\rho_t$ is a probability density function on the unit circle $\bbs^1$. The system can be reduced to the following Kuramoto-Sakaguchi system:
\begin{align*}
\begin{cases}
\partial_t\rho_t+\partial_\theta(\mathcal{V}[\rho_t]\rho_t)=0,\quad t>0,\\
\displaystyle\mathcal{V}[\rho]_t(\theta):=\kappa \int_{\mathbb{T}}\sin(\theta_*-\theta)\rho_t(\theta_*)d\theta_*,\quad \theta\in\bbt^1.
\end{cases}
\end{align*}
We substitute $x(\theta)=(\cos\theta, \sin\theta)$ to obtain
\[
\|x_1-x_2\|^2=2(1-\cos(\theta_1-\theta_2))=4\sin^2\left(\frac{\theta_1-\theta_2}{2}\right),
\]
and for the simplicity we abuse a notation $\rho_t(\theta):=\rho_t((\cos\theta, \sin\theta))$. From this, we can express $\mathcal{H}_p[\rho_t]$ as follows:
\begin{align*}
\mathcal{H}_p[\rho_t]&=\int_{(\bbs^1)^4}\left(\frac{\|x_1-x_2\|\cdot \|x_3-x_4\|}{\|x_2-x_3\|\cdot \|x_4-x_1\|}\right)^{2p}\rho_t(x_1)\rho_t(x_2)\rho_t(x_3)\rho_t(x_4)d\sigma_1d\sigma_2d\sigma_3d\sigma_4\\
&=\int_{[-\pi, \pi]^4}\left(\frac{\sin\left(\frac{\theta_1-\theta_2}{2}\right)\sin\left(\frac{\theta_3-\theta_4}{2}\right)}{\sin\left(\frac{\theta_2-\theta_3}{2}\right)\sin\left(\frac{\theta_4-\theta_1}{2}\right)}\right)^{2p}\rho_t(\theta_1)\rho_t(\theta_2)\rho_t(\theta_3)\rho_t(\theta_4)d\theta_1d\theta_2d\theta_3d\theta_4.
\end{align*}
Now we consider $p=\frac{1}{2}$. If $\rho_t$ is the uniform distribution on the circle, from
\[
\frac{2}{\pi}|x|\leq |\sin x|\leq |x|\quad \forall~x\in [-\pi, \pi]
\]
then we have
\begin{align*}
\mathcal{H}_{\frac{1}{2}}[\rho_t]&=\left(\frac{1}{2\pi}\right)^4\int_{[-\pi, \pi]^4}\left|\frac{\sin\left(\frac{\theta_1-\theta_2}{2}\right)\sin\left(\frac{\theta_3-\theta_4}{2}\right)}{\sin\left(\frac{\theta_2-\theta_3}{2}\right)\sin\left(\frac{\theta_4-\theta_1}{2}\right)}\right|d\theta_1d\theta_2d\theta_3d\theta_4\\
&\geq\left(\frac{1}{2\pi}\right)^4\left(\frac{2}{\pi}\right)^2\int_{[-\pi, \pi]^4}\left|\frac{(\theta_1-\theta_2)(\theta_3-\theta_4)}{(\theta_2-\theta_3)(\theta_4-\theta_1)}\right|d\theta_1d\theta_2d\theta_3d\theta_4.
\end{align*}
If we consider the neighborhood of singularities $\theta_2=\theta_3$ or $\theta_1=\theta_4$, then the integration diverges to infinity. So, $\mathcal{H}_{\frac{1}{2}}[\rho_t]=\infty$ and it does not exist in this case. This implies that, for some $p$, $\mathcal{H}_p[\rho_t]$ does not exist. We need to find a condition of $p$ which guarantee the existence of $\mathcal{H}_p[\rho_t]$. In the rest of this subsection, we do not consider the existence of $\mathcal{H}_p[\rho_t]$, but we only prove that $\mathcal{H}_p$ and $\mathcal{H}_{p, k}$ are constant of motion functional. We remain details on the existence in Section \ref{sec:4}.\\

\subsubsection{$\mathcal{H}_p$ and $\mathcal{H}_{p, k}$ are constant of motion functionals}\label{sec:3.2.2}
Now we prove that $\mathcal{H}_p$ and $\mathcal{H}_{p, k}$ are constants of motion, if they exist.
Let $\rho_t$ be a solution of system \eqref{C-(1)}. Recall that
\[
\rho_t=M_t\#\rho_0.
\]
From simple calculations, we have
\begin{align*}
\mathcal{H}_p[\rho_t]=&\int_{(\bbs^d)^4}\mathcal{C}(x_1, x_2, x_3, x_4)^p\rho_t(x_1)\rho_t(x_2)\rho_t(x_3)\rho_t(x_4)d\sigma_1d\sigma_2d\sigma_3d\sigma_4\\
=&\int_{(\bbs^d)^4}\mathcal{C}(M_t(x_1), M_t(x_2), M_t(x_3), M_t(x_4))^p\rho_0(x_1)\rho_0(x_2)\rho_0(x_3)\rho_0(x_4)d\sigma_1d\sigma_2d\sigma_3d\sigma_4\\
=&\int_{(\bbs^d)^4}\mathcal{C}(x_1, x_2, x_3, x_4)^p\rho_0(x_1)\rho_0(x_2)\rho_0(x_3)\rho_0(x_4)d\sigma_1d\sigma_2d\sigma_3d\sigma_4=\mathcal{H}_p[\rho_0].
\end{align*}
In the second equality, we used the simple fact of the push forward of measure, and in third equality, we used that $M_t$ leaves $\mathcal{C}$ invariant. This implies that $\mathcal{H}_p[\rho_t]$ is a constant of motion for system \eqref{C-(1)} if it exists. Similarly, $\mathcal{H}_{p, k}[\rho_t]$ is also a constant of motion for system \eqref{C-(1)}, since $M_t$ leaves $\mathcal{C}_k$ invariant. We can summarize these results as the following theorem.
\begin{theorem}\label{Thm3.2}
Let $\rho_t$ be a solution of system \eqref{C-(1)}. Then $\mathcal{H}_p$ and $\mathcal{H}_{p, k}$ are constant of motion functionals for all $p\in\bbr$ and an integer $k\geq2$ if they exist.
\end{theorem}

\section{Existence of $\mathcal{H}_p$ and $\mathcal{H}_{p, k}$}\label{sec:4}
\setcounter{equation}{0}

In this section, we find necessary and sufficient conditions for $p$ which guarantee the existence of $\mathcal{H}_p[\rho_t]$. We define the uniform distribution on the sphere $\bbs^d$ as follows:
\[
\rho_{\mathrm{uni}}\equiv\frac{1}{|\bbs^d|},
\]
where $|\bbs^d|$ is a $d$-dimensional Hausdorff measure of $\bbs^d$. The main goal of this section is to show the following three statements ($\mathcal{S}1$)-($\mathcal{S}3$) are equivalent:\\

($\mathcal{S}1$) $p$ and $d$ satisfy $p\in\left(-\frac{d}{2},\frac{d}{2}\right)$,\vspace{0.2cm}

($\mathcal{S}2$) $\mathcal{H}_p[\rho_{\mathrm{uni}}]$ exists for uniform distribution $\rho_{\mathrm{uni}}$,\vspace{0.2cm}

($\mathcal{S}3$) $\mathcal{H}_p[\rho]$ exists for an arbitrary probability density fnction $\rho\in\mathcal{P}(\bbs^d)$.\\

The following flow chart shows the relation between three statements and their connections:
\[
(\mathcal{S}1)\quad \substack{\displaystyle\xRightarrow[]{\text{Theorem \ref{Thm4.1}}}\\\displaystyle\xLeftarrow[\text{Theorem \ref{Thm4.2}}]{}}\quad(\mathcal{S}2)\quad \substack{\displaystyle\xRightarrow[]{\text{Lemma \ref{Lem4.1}}}\\\displaystyle\xLeftarrow[\text{Theorem \ref{Thm4.3}}]{}}\quad (\mathcal{S}3).
\]

\begin{theorem}\label{Thm4.1}
Let $p$ be a real number and $d$ be the dimension of the domain $\bbs^d$. If $p$ and $d$ satisfy the following relation:
\[
 p\in\left(- \frac{d}{2}, \frac{d}{2}\right),
\]
then $\mathcal{H}_p[\rho_{\mathrm{uni}}]$ exists.
\end{theorem}

\begin{proof}
From Proposition \ref{Pro3.1} (2), we know that $\mathcal{H}_{-p, k}=\mathcal{H}_{p, k}$. So we consider only the case when $0\leq p<\frac{d}{2}$. We have the following calculation
\begin{align*}
\mathcal{H}_p[\rho_{\mathrm{uni}}]&=\int_{(\bbs^d)^4}\left(\frac{\|x_1-x_2\|\cdot \|x_3-x_4\|}{\|x_2-x_3\|\cdot \|x_4-x_1\|}\right)^{2p}\rho(x_1)\rho(x_2)\rho(x_3)\rho(x_4)d\sigma_{1}d\sigma_{2}d\sigma_{3}d\sigma_{4}\\
&=\frac{1}{|\bbs^d|^4}\int_{(\bbs^d)^4}\left(\frac{\|x_1-x_2\|\cdot \|x_3-x_4\|}{\|x_2-x_3\|\cdot \|x_4-x_1\|}\right)^{2p}d\sigma_{1}d\sigma_{2}d\sigma_{3}d\sigma_{4}\\
&\leq \frac{1}{|\bbs^d|^4}\int_{(\bbs^d)^4}\left(\frac{4}{\|x_2-x_3\|\cdot \|x_4-x_1\|}\right)^{2p}d\sigma_{1}d\sigma_{2}d\sigma_{3}d\sigma_{4}\\
&=\frac{4^{2p}}{|\bbs^d|^4}\left(\int_{(\bbs^d)^2}\frac{1}{\|x_2-x_3\|^{2p}}d\sigma_{2}d\sigma_{3}\right)^2.
\end{align*}
This implies that if an integral
\[
\int_{(\bbs^d)^2}\frac{1}{\|x_2-x_3\|^{2p}}d\sigma_{2}d\sigma_{3}
\]
is finite, then $\mathcal{H}_p[\rho_{\mathrm{uni}}]$ exists. From the rotation symmetry of $\bbs^d$, we have
\begin{align*}
\int_{(\bbs^d)^2}\frac{1}{\|x_2-x_3\|^{2p}}d\sigma_{2}d\sigma_{3}=\int_{(\bbs^d)^2}\frac{1}{\|e-x_3\|^{2p}}d\sigma_{2}d\sigma_{3}=|\bbs^d|\int_{\bbs^d}\frac{1}{\|e-x\|^{2p}}d\sigma_{x}
\end{align*}
for any fixed vector $e\in\bbs^d$. To make a concrete calculation, we set
\[
e=(1, \underbrace{0, 0, \cdots, 0}_{d-\text{times}})\in\bbr^{d+1}.
\]
We also parametrize $x\in\bbs^d$ with $\theta\in[0, \pi]$ and $y=([y]_1, [y]_2, \cdots, [y]_d)\in\bbs^{d-1}\subset\bbr^d$ as follows:
\begin{align*}
&x=(\cos\theta, \sin\theta y)=(\cos\theta, \sin\theta [y]_1, \sin\theta [y]_2, \cdots, \sin\theta [y]_d).
\end{align*}
Then we have
\[
\|e-x\|^2=2-2\langle e, x\rangle=2(1-\cos\theta)=4\sin^2\frac{\theta}{2}.
\]
Since the volume form $\omega_d$ of $\bbs^d$ embedded in $\bbr^{d+1}$ can be expressed as
\begin{align}\label{D-0}
\omega_d=\sum_{j=1}^{d+1}(-1)^{j-1}[x]_jd[x]_1\wedge \cdots\wedge d[x]_{j-1}\wedge d[x]_{j+1}\wedge \cdots\wedge d[x]_{d+1},
\end{align}
if we substitute
\[
[x]_1=\cos\theta,\quad [x]_j=\sin\theta~ [y]_{j-1},\quad\forall~2\leq j\leq d+1,
\]
into \eqref{D-0}, we have
\begin{align*}
\omega_d&=[x]_1d[x]_2\wedge\cdots\wedge [x]_{d+1}-\sum_{j=1}^d (-1)^{j-1} [x]_{j}d[x]_1\wedge \cdots\wedge d[x]_j \wedge d[x]_{j+2}\wedge\cdots \wedge d[x]_{d+1}\\
&=\cos\theta \sin^{d}\theta d[y]_1\wedge \cdots\wedge d[y]_d-\sum_{j=1}^d\sin^{d-1}\theta [y]_jd\theta\wedge d[y]_1\wedge\cdots d[y]_{j-1}\wedge d[y]_{j+1}\wedge\cdots\wedge d[y]_d\\
&=0-\sin^{d-1}\theta d\theta\wedge \omega_{d-1}=-\sin^{d-1}\theta d\theta\wedge \omega_{d-1}.
\end{align*}
Here, we used $d[y]_1\wedge \cdots\wedge d[y]_d=0$, since $[y]_1d[y]_1+\cdots+[y]_dd[y]_d=0$ holds.
This implies that
\begin{align*}
d\sigma_x=(\sin\theta)^{d-1}d\sigma_y d\theta.
\end{align*}
Again, this yields
\begin{align*}
\int_{\bbs^d}\frac{1}{\|e-x\|^{2p}}d\sigma_{x}&=\int_0^\pi\int_{\bbs^{d-1}}\frac{1}{\left(4\sin^2\frac{\theta}{2}\right)^p}(\sin\theta)^{d-1}d\sigma_yd\theta\\
&=|\bbs^{d-1}|\int_0^\pi\frac{1}{\left(4\sin^2\frac{\theta}{2}\right)^p}(\sin\theta)^{d-1}d\theta\\
&=|\bbs^{d-1}|\int_0^\pi\frac{2^{d-1}\left(\sin\frac{\theta}{2}\right)^{d-1}\left(\cos\frac{\theta}{2}\right)^{d-1}}{\left(4\sin^2\frac{\theta}{2}\right)^p}d\theta\\
&=|\bbs^{d-1}|\int_0^\pi2^{d-2p-1}\left(\sin\frac{\theta}{2}\right)^{d-2p-1}\left(\cos\frac{\theta}{2}\right)^{d-1}d\theta.
\end{align*}
The last term of the above integral exists if and only if the following limit
\begin{align}\label{D-0-1}
\int_0^\delta \theta^{d-2p-1}d\theta
\end{align}
exists for some $\delta\in(0, \pi)$. From the simple calculus, we know that a limit \eqref{D-0-1} exists if and only if 
\[
d-2p>0.
\]
This yields, if
\[
0\leq p< \frac{d}{2},
\]
then $\mathcal{H}_p[\rho_{\mathrm{uni}}]$ exists. By the symmetry $\mathcal{H}_{-p}=\mathcal{H}_{p}$, $\mathcal{H}_{p}[\rho_{\mathrm{uni}}]$ exists if $p\in\left(-\frac{d}{2},\frac{d}{2}\right)$.
\end{proof}

From Theorem \ref{Thm4.1}, we obtain the sufficient condition for $p$ which guarantees the existence of $\mathcal{H}_p[\rho_{\mathrm{uni}}]$. Now, we study a necessary condition part. For convenience, we define
\[
B_r(x):=\{y\in\bbs^d: \|x-y\|<r\}
\]
for any $r>0$ and $x\in\bbs^d$.

\begin{theorem}\label{Thm4.2}
Let $p$ be a real number and $d$ be the dimension of the domain $\bbs^d$. If $\mathcal{H}_p[\rho_{\mathrm{uni}}]$ exist, then $p$ and $d$ satisfy
\[
p\in \left(-\frac{d}{2},\frac{d}{2}\right).
\]
\end{theorem}
\begin{proof}
Since $\mathcal{H}_p=\mathcal{H}_{-p}$, without loss of generality, we assume $p\geq0$. We define the neighborhood of $e$ and $-e$ as follows:
\[
\mathcal{S}_1=B_\varepsilon(e),\quad \mathcal{S}_2=B_\varepsilon(-e),
\]
where $e=(1, \underbrace{0, 0, \cdots, 0}_{d-\text{times}})\in\bbr^{d+1}$ and $0<\varepsilon<\sqrt{2}$. Then $\mathcal{S}_1$ and $\mathcal{S}_2$ are disjoint sets and the distance between $\mathcal{S}_1$ and $\mathcal{S}_2$ can be calculated as
\[
\mathrm{dist}(\mathcal{S}_1,\mathcal{S}_2)=\inf_{\substack{x\in\mathcal{S}_1,\\ y\in\mathcal{S}_2}}\mathrm{dist}(x, y)=2-\varepsilon^2.
\]
From this relation, we have the following calculation:
\begin{align*}
\mathcal{H}_p[\rho_{\mathrm{uni}}]&=\frac{1}{|\bbs^d|^4}\int_{(\bbs^d)^4}\left(\frac{\|x_1-x_2\|\cdot \|x_3-x_4\|}{\|x_2-x_3\|\cdot \|x_4-x_1\|}\right)^{2p}d\sigma_{1}d\sigma_{2}d\sigma_{3}d\sigma_{4}\\
&\geq\frac{1}{|\bbs^d|^4} \int_{\mathcal{S}_1}\int_{\mathcal{S}_2}\int_{\mathcal{S}_2}\int_{\mathcal{S}_1}\left(\frac{\|x_1-x_2\|\cdot \|x_3-x_4\|}{\|x_2-x_3\|\cdot \|x_4-x_1\|}\right)^{2p}d\sigma_{1}d\sigma_{2}d\sigma_{3}d\sigma_{4}\\
&\geq\frac{1}{|\bbs^d|^4}(2-\varepsilon^2)^2 \int_{\mathcal{S}_1}\int_{\mathcal{S}_2}\int_{\mathcal{S}_2}\int_{\mathcal{S}_1}\left(\frac{1}{\|x_2-x_3\|\cdot \|x_4-x_1\|}\right)^{2p}d\sigma_{1}d\sigma_{2}d\sigma_{3}d\sigma_{4}\\
&=\frac{1}{|\bbs^d|^4}(2-\varepsilon^2)^2\left( \int_{(\mathcal{S}_1)^2}\frac{1}{\|x-y\|^{2p}}d\sigma_{x}d\sigma_{y}\right)^2.
\end{align*}
In the last equality, we used the symmetry of $\mathcal{S}_1$ and $\mathcal{S}_2$. Since we assumed that $\mathcal{H}_p[\rho_{\mathrm{uni}}]$ exists, we know that a limit
\begin{align}\label{D-0-3}
\int_{(\mathcal{S}_1)^2}\frac{1}{\|x-y\|^{2p}}d\sigma_{x}d\sigma_{y}
\end{align}
exists. Now, for $\delta>0$, we define
\[
\mathcal{S}_1^\delta:=\{y:B_\delta(y)\subset \mathcal{S}_1\}.
\]
Then, for sufficiently small $\delta>0$, $\mathcal{S}_1^\delta$ is a non-empty set. This definition yields the following inequality:
\begin{align}\label{D-0-4}
\int_{(\mathcal{S}_1)^2}\frac{1}{\|x-y\|^{2p}}d\sigma_{x}d\sigma_{y}\geq \int_{\mathcal{S}_1}\left(\int_{\mathcal{S}_1^\delta}\frac{1}{\|x-y\|^{2p}}d\sigma_{x}\right)d\sigma_{y}.
\end{align}
We have the following calculation:
\begin{align}
\begin{aligned}\label{D-0-5}
\int_{\mathcal{S}_1^\delta}\frac{1}{\|x-y\|^{2p}}d\sigma_{x}&=\int_{B_\delta (y)}\frac{1}{\|x-y\|^{2p}}d\sigma_{x}+\int_{\mathcal{S}_1^\delta\cap (B_\delta(y))^c}\frac{1}{\|x-y\|^{2p}}d\sigma_{x}\\
&=\int_{B_\delta (e)}\frac{1}{\|x-e\|^{2p}}d\sigma_{x}
+\int_{\mathcal{S}_1^\delta\cap (B_\delta(y))^c}\frac{1}{\|x-y\|^{2p}}d\sigma_{x}\\
&\geq\int_{B_\delta (e)}\frac{1}{\|x-e\|^{2p}}d\sigma_{x}.
\end{aligned}
\end{align}
Here, we used the symmetry of $\bbs^d$ in the second equality. Finally, we combine relations \eqref{D-0-3}, \eqref{D-0-4}, and \eqref{D-0-5} to get 
\begin{align}\label{D-0-6}
\int_{B_\delta (e)}\frac{1}{\|x-e\|^{2p}}d\sigma_{x}<\infty.
\end{align}
By a similar argument that we used in the proof of Theorem \ref{Thm4.1}, we have
%
\begin{align*}
\int_{B_\delta (e)}\frac{1}{\|x-e\|^{2p}}d\sigma_{x}&=\int_0^{2\sin^{-1}\left(\frac{\delta}{2}\right)}\int_{\bbs^{d-1}}\frac{1}{\left(4\sin^2\frac{\theta}{2}\right)^p}(\sin\theta)^{d-1}d\sigma_yd\theta\\
&=|\bbs^{d-1}|\int_0^{2\sin^{-1}\left(\frac{\delta}{2}\right)}2^{d-2p-1}\left(\sin\frac{\theta}{2}\right)^{d-2p-1}\left(\cos\frac{\theta}{2}\right)^{d-1}d\theta
\end{align*}
and the last integral exists if and only if $0\leq p<\frac{d}{2}$. So we know that \eqref{D-0-6} implies 
\[
0\leq p<\frac{d}{2}.
\]
By the symmetry of $\mathcal{H}_p$ (Proposition \ref{Pro3.1}), we have 
\[
p\in\left(-\frac{d}{2},\frac{d}{2}\right).
\]
\end{proof}

\begin{remark}\label{Rem4.1}
From Theorems \ref{Thm4.1} and \ref{Thm4.2}, we know that
\[
p\in\left(-\frac{d}{2},\frac{d}{2}\right)\quad\Longleftrightarrow\quad \mathcal{H}_p[\rho_{\mathrm{uni}}]\quad\text{exists}.
\]
\end{remark}

So far, we only studied the existence of $\mathcal{H}_p[\rho_{\mathrm{uni}}]$. From now on, we find an equivalent condition for $p$ to the existence of $\mathcal{H}_p[\rho]$ for general probability density function $\rho\in\mathcal{P}(\bbs^d)$.

\begin{lemma}\label{Lem4.1}
Let $p$ be a real number and $d$ be the dimension of the domain $\bbs^d$. If $\mathcal{H}_p[\rho_{\mathrm{uni}}]$ exists, then $\mathcal{H}_p[\rho]$ exists for all $\rho\in\mathcal{P}(\bbs^d)$.
\end{lemma}

\begin{proof}
Since $\rho$ is a continuous function defined on the compact domain $\bbs^d$, we can define $M=\max_{x\in\bbs^d}\rho(x)$. From the following inequality:
\begin{align*}
\mathcal{H}_p[\rho]
&=\int_{(\bbs^d)^4}\left(\frac{\|x_1-x_2\|\cdot \|x_3-x_4\|}{\|x_2-x_3\|\cdot \|x_4-x_1\|}\right)^{2p}\rho(x_1)\rho(x_2)\rho(x_3)\rho(x_4)d\sigma_{1}d\sigma_{2}d\sigma_{3}d\sigma_{4}\\
&\leq\int_{(\bbs^d)^4}\left(\frac{\|x_1-x_2\|\cdot \|x_3-x_4\|}{\|x_2-x_3\|\cdot \|x_4-x_1\|}\right)^{2p}M^4d\sigma_{1}d\sigma_{2}d\sigma_{3}d\sigma_{4}\\
&\leq \left(M|\bbs^d|\right)^4\int_{(\bbs^d)^4}\left(\frac{\|x_1-x_2\|\cdot \|x_3-x_4\|}{\|x_2-x_3\|\cdot \|x_4-x_1\|}\right)^{2p}\frac{1}{|\bbs^d|^4}d\sigma_{1}d\sigma_{2}d\sigma_{3}d\sigma_{4}\\
&= \left(M|\bbs^d|\right)^4\mathcal{H}_p[\rho_{\mathrm{uni}}].
\end{align*}
So $\mathcal{H}_p[\rho]$ exists, because we assumed that $\mathcal{H}_p[\rho_{\mathrm{uni}}]$ exists.
\end{proof}

\begin{corollary}
Let $p$ be a real number and $d$ be the dimension of the domain $\bbs^d$. If $p$ and $d$ satisfy
\[
p\in\left(-\frac{d}{2},\frac{d}{2}\right),
\]
then $\mathcal{H}_p[\rho]$ exists for all $\rho\in\mathcal{P}(\bbs^d)$.
\end{corollary}

\begin{proof}
From Theorem \ref{Thm4.2} and the given condition $|p|<\frac{d}{2}$ imply the existence of $\mathcal{H}_p[\rho_{\mathrm{uni}}]$. Again, Lemma \ref{Lem4.1} and existence of $\mathcal{H}_p[\rho_{\mathrm{uni}}]$ yield the existence of $\mathcal{H}_p[\rho]$.
\end{proof}

\begin{theorem}\label{Thm4.3}
Let $p$ be a real number. If there exists a probability density function $\rho\in\mathcal{P}(\bbs^d)$ which makes $\mathcal{H}_p[\rho]$ exists, then $\mathcal{H}_p[\rho_{\mathrm{uni}}]$ exists.
\end{theorem}

\begin{proof}
We assume that there exists $\rho\in\mathcal{P}(\bbs^d)$ which satisfies $\mathcal{H}_p[\rho]<\infty$. Without loss of generality, we assume that $p\geq0$. Since $\bbs^d$ is compact and $\rho$ is continuous, there exist $x_M\in\bbs^d$ satisfies
\[
\rho(x_M)=\max_{x\in\bbs^d}\rho(x).
\]
Since $\max_{x\in\bbs^d}\rho(x)$ is a positive real number, we can define the following non-empty closed subset of $\bbs^d$ as follows:
\[
\mathcal{S}:=\left\{x\in\bbs^d: \rho(x)\geq \frac{1}{2}\max_{x\in\bbs^d}\rho(x)\right\}.
\]
Then there exists $\varepsilon>0$ such that 
\[
B_\varepsilon(x_M)\subset \mathcal{S},
\]
since $\rho$ is a continuous function. Under this setting, we have
\begin{align*}
\infty>&\int_{(\bbs^d)^4}\left(\frac{\|x_1-x_2\|\cdot \|x_3-x_4\|}{\|x_2-x_3\|\cdot \|x_4-x_1\|}\right)^{2p}\rho(x_1)\rho(x_2)\rho(x_3)\rho(x_4)d\sigma_{1}d\sigma_{2}d\sigma_{3}d\sigma_{4}\\
\geq&\int_{B_\varepsilon(x_M)^4}\left(\frac{\|x_1-x_2\|\cdot \|x_3-x_4\|}{\|x_2-x_3\|\cdot \|x_4-x_1\|}\right)^{2p}\rho(x_1)\rho(x_2)\rho(x_3)\rho(x_4)d\sigma_{1}d\sigma_{2}d\sigma_{3}d\sigma_{4}\\
\geq&\left(\frac{1}{2}\rho(x_M)\right)^4\int_{B_\varepsilon(x_M)^4}\left(\frac{\|x_1-x_2\|\cdot \|x_3-x_4\|}{\|x_2-x_3\|\cdot \|x_4-x_1\|}\right)^{2p}d\sigma_{1}d\sigma_{2}d\sigma_{3}d\sigma_{4}.
\end{align*}
For sufficiently small $\tilde{\varepsilon}>0$, we know that there exists two points $y_1$ and $y_2$ such that
\[
y_1, y_2\in B_\varepsilon(x_M),\quad B_{2\tilde{\varepsilon}}(y_1),B_{2\tilde{\varepsilon}}(y_2)\subset B_\varepsilon(x_M),\quad B_{2\tilde{\varepsilon}}(y_1)\cap B_{2\tilde{\varepsilon}}(y_2)=\emptyset.
\]
Then we can easily obtain that
\begin{align}\label{D-2}
\mathrm{dist}(B_{\tilde{\varepsilon}}(y_1), B_{\tilde{\varepsilon}}(y_2))\geq 2\tilde{\varepsilon},
\end{align}
and
\begin{align}\label{D-3}
B_{\tilde{\varepsilon}}(y_1)\times B_{\tilde{\varepsilon}}(y_2)\times B_{\tilde{\varepsilon}}(y_2)\times B_{\tilde{\varepsilon}}(y_1)\subset B_{\varepsilon}(x_M)^4.
\end{align}
From two relations \eqref{D-2} and \eqref{D-3}, we get
\begin{align*}
&\int_{B_\varepsilon(x_M)^4}\left(\frac{\|x_1-x_2\|\cdot \|x_3-x_4\|}{\|x_2-x_3\|\cdot \|x_4-x_1\|}\right)^{2p}d\sigma_{1}d\sigma_{2}d\sigma_{3}d\sigma_{4}\\
\geq&\int_{B_{\tilde{\varepsilon}}(y_1)}\int_{B_{\tilde{\varepsilon}}(y_2)}\int_{B_{\tilde{\varepsilon}}(y_2)}\int_{B_{\tilde{\varepsilon}}(y_1)}\left(\frac{\|x_1-x_2\|\cdot \|x_3-x_4\|}{\|x_2-x_3\|\cdot \|x_4-x_1\|}\right)^{2p}d\sigma_{1}d\sigma_{2}d\sigma_{3}d\sigma_{4}\\
\geq&\int_{B_{\tilde{\varepsilon}}(y_1)}\int_{B_{\tilde{\varepsilon}}(y_2)}\int_{B_{\tilde{\varepsilon}}(y_2)}\int_{B_{\tilde{\varepsilon}}(y_1)}\left(\frac{(2\tilde{\varepsilon})^2}{\|x_2-x_3\|\cdot \|x_4-x_1\|}\right)^{2p}d\sigma_{1}d\sigma_{2}d\sigma_{3}d\sigma_{4}\\
=&\left(\int_{B_{\tilde{\varepsilon}}(y_2)}\int_{B_{\tilde{\varepsilon}}(y_2)}\left(\frac{2\tilde{\varepsilon}}{\|x_2-x_3\|}\right)^{2p}d\sigma_{2}d\sigma_{3}\right)^2.
\end{align*}
The above calculation yields
\[
\int_{B_{\tilde{\varepsilon}}(y_2)}\int_{B_{\tilde{\varepsilon}}(y_2)}\frac{1}{\|x_2-x_3\|^{2p}}d\sigma_{2}d\sigma_{3}<\infty.
\]
Since
\[
\bigcup_{x_3\in B_{\tilde{\varepsilon}/2}(y_2)}\Big(\{x_3\}\times B_{\tilde{\varepsilon}/2}(x_3)\Big)\subset B_{\tilde{\varepsilon}}(y_2)\times B_{\tilde{\varepsilon}}(y_2),
\]
we have
\begin{align*}
&\int_{B_{\tilde{\varepsilon}}(y_2)}\int_{B_{\tilde{\varepsilon}}(y_2)}\frac{1}{\|x_2-x_3\|^{2p}}d\sigma_{2}d\sigma_{3}\geq \int_{B_{\tilde{\varepsilon}/2}(y_2)}\int_{B_{\tilde{\varepsilon}/2}(x_3)}\frac{1}{\|x_2-x_3\|^{2p}}d\sigma_{2}d\sigma_{3}\\
=& |B_{\tilde{\varepsilon}/2}(y_2)|\int_{B_{\tilde{\varepsilon}/2}(x_3)}\frac{1}{\|x_2-x_3\|^{2p}}d\sigma_{2}
=|B_{\tilde{\varepsilon}/2}(y_2)|\int_{B_{\tilde{\varepsilon}/2}(e)}\frac{1}{\|x-e\|^{2p}}d\sigma_{x}.
\end{align*}
Finally, we can obtain that
\[
\int_{B_{\tilde{\varepsilon}/2}(e)}\frac{1}{\|x-e\|^{2p}}d\sigma_{x}<\infty,
\]
and from a similar argument that we used in Theorem \ref{Thm4.2}, we can conclude that $0\leq p<\frac{d}{2}$. By Theorem \ref{Thm4.1}, we know that $\mathcal{H}_p[\rho_{\mathrm{uni}}]<\infty$ and it exists. So we have the desired result.
\end{proof}

\begin{remark}\label{Rem4.2}
Let a probability density function $\rho\in\mathcal{P}(\bbs^d)$ be given. Then $\mathcal{H}_p[\rho]<\infty$ if and only if $p\in\left(-\frac{d}{2},\frac{d}{2}\right)$.
\end{remark}

\begin{corollary}
Let $\mu$ be a finite measure defined on $\left[-\frac{d}{2}+\delta, \frac{d}{2}-\delta\right]$ for some $\delta\in\left(0, \frac{d}{2}\right)$. Then the following functional
\[
\mathcal{F}[\rho]=\int_{-d/2+\delta}^{d/2-\delta}\mathcal{H}_p[\rho] d\mu(p)
\]
is also a constant of motion functional of system \eqref{C-(1)}.
\end{corollary}

\begin{proof}
Since $\mathcal{H}_p$ is a constant of motion functional for all $p\in\left[-\frac{d}{2}+\delta, \frac{d}{2}-\delta\right]$ and $\mathcal{F}$ is a kind of linear combination of them, $\mathcal{F}$ is a constant of motion functional of system \eqref{C-(1)}.
\end{proof}

From a similar argument in Theorem \ref{Thm4.3}, we have the following theorem.
\begin{theorem}
Let a probability density function $\rho$ on $\bbs^d$ and an integer $k\geq2$ be given. Then $\mathcal{H}_{p, k}[\rho]<\infty$ if and only if $p\in\left(-\frac{d}{2},\frac{d}{2}\right)$.
\end{theorem}

\section{Application of constant of motion functionals}\label{sec:5}
\setcounter{equation}{0}

In this section, we provide an application of constant of motion functionals for system \eqref{C-(1)}. We begin with the following idea. If $y_1$ and $y_2$ are distinct points on $\bbs^d$ and $0<q<1$ are given, then we have
\begin{align}\label{E-0}
\mathcal{H}_p[q\delta_{y_1}+(1-q)\delta_{y_2}]=\infty.
\end{align}
We can check this by the following formal calculation. We set $\tilde{\rho}:=q\delta_{y_1}+(1-q)\delta_{y_2}$, then
\begin{align*}
\mathcal{H}_p[q\delta_{y_1}+(1-q)\delta_{y_2}]&=\int_{(\bbs^d)^4}\left(\frac{\|x_1-x_2\|\cdot\|x_3-x_4\|}{\|x_2-x_3\|\cdot\|x_4-x_1\|}\right)^{2p}\tilde{\rho}(x_1)\tilde{\rho}(x_2)\tilde{\rho}(x_3)\tilde{\rho}(x_4)d\sigma_1 d\sigma_2d\sigma_3d\sigma_4\\
&\geq\int_{\{(y_1, y_2, y_2, y_1)\}}\left(\frac{\|x_1-x_2\|\cdot\|x_3-x_4\|}{\|x_2-x_3\|\cdot\|x_4-x_1\|}\right)^{2p}\tilde{\rho}(x_1)\tilde{\rho}(x_2)\tilde{\rho}(x_3)\tilde{\rho}(x_4)d\sigma_1 d\sigma_2d\sigma_3d\sigma_4\\
&= q^2(1-q)^2\times \infty=\infty.
\end{align*}
So we proved \eqref{E-0}, and this can be used to discuss the instability of a distribution $\tilde{\rho}$. For any $\varepsilon>0$, we can choose a continuous probability density function $\rho\in\mathcal{P}(\bbs^d)$ which satisfies:
\[
W_2(\tilde{\rho}, \rho)<\varepsilon,
\]
where $W_2$ is the Wasserstein-2 metric. Recall that 
\[
\mathcal{H}_p[\tilde{\rho}]=\infty\quad\text{and}\quad \mathcal{H}_p[\rho]<\infty.
\]
Although the Wasserstein-2 distance between $\tilde{\rho}$ and $\rho$ is small enough, we have $|\mathcal{H}_p[\tilde{\rho}]-\mathcal{H}_p[\rho]|=\infty$. Since $\mathcal{H}_p$ is a constant of motion functional, we can understand $\tilde{\rho}$ as an unstable state. We use this fact to show the instability of the bipolar state.

\begin{lemma}\label{Lem5.1}
Let $\rho_t\in\mathcal{P}(\bbs^d)$ be a family of density functions on $t\in[0, \infty)$ which is continuous along the time variable and satisfy
\begin{align}\label{E-0-0}
\lim_{t\to\infty}\int_{B_{\varepsilon}(c_k(t))}\rho_t(x)d\sigma_x\geq\delta_k,\quad k\in\{1, 2\},\quad \forall~\varepsilon\in(0, \varepsilon_0),
\end{align}
for fixed positive constants $\delta_1,\delta_2,\varepsilon_0$, and continuous curves on the unit sphere $\{c_k:[0, \infty)\to\bbs^d~|~k=1, 2\}$. If two curves $c_1$ and $c_2$ satisfy
\[
\|c_1(t)-c_2(t)\|\geq\delta,\quad \forall~t\geq0,
\]
for a positive real number $\delta$, then we have 
\[
\lim_{t\to\infty}\mathcal{H}_p[\rho_t]=\infty.
\]
\end{lemma}

\begin{proof}
By direct calculations, we have
\begin{align}
\begin{aligned}\label{E-0-1}
\mathcal{H}_p[\rho_t]=&\int_{(\bbs^d)^4}\left(\frac{\|x_1-x_2\|\cdot \|x_3-x_4\|}{\|x_2-x_3\|\cdot \|x_4-x_1\|}\right)^{2p}\rho_t(x_1)\rho_t(x_2)\rho_t(x_3)\rho_t(x_4)d\sigma_{1}d\sigma_{2}d\sigma_{3}d\sigma_{4}\\
\geq&\sum_{k_1, k_2, k_3, k_4=1}^2\int_{B_\varepsilon(c_{k_1})\times B_\varepsilon(c_{k_2})\times B_\varepsilon(c_{k_3})\times B_\varepsilon(c_{k_4})}\left(\frac{\|x_1-x_2\|\cdot \|x_3-x_4\|}{\|x_2-x_3\|\cdot \|x_4-x_1\|}\right)^{2p}\\
&\hspace{5cm}\times\rho_t(x_1)\rho_t(x_2)\rho_t(x_3)\rho_t(x_4)d\sigma_{1}d\sigma_{2}d\sigma_{3}d\sigma_{4}\\
\geq&\sum_{\substack{k_1, k_2=1\\ k_1\neq k_2}}^2\int_{B_\varepsilon(c_{k_1})\times B_\varepsilon(c_{k_2})\times B_\varepsilon(c_{k_2})\times B_\varepsilon(c_{k_1})}\left(\frac{\|x_1-x_2\|\cdot \|x_3-x_4\|}{\|x_2-x_3\|\cdot \|x_4-x_1\|}\right)^{2p}\\
&\hspace{5cm}\times\rho_t(x_1)\rho_t(x_2)\rho_t(x_3)\rho_t(x_4)d\sigma_{1}d\sigma_{2}d\sigma_{3}d\sigma_{4}1.
\end{aligned}
\end{align}
From the triangle inequality, we have following inequalities:
\[
\begin{cases}
x, y\in B_\varepsilon(c_1)\quad\text{or}\quad x, y\in B_\varepsilon(c_2)\quad\Rightarrow\quad \|x-y\|\leq 2\varepsilon,\\
x\in B_\varepsilon(c_1),\quad y\in B_{\varepsilon}(c_2)\quad\Rightarrow\quad \|x-y\|\geq \|c_1-c_2\|-2\varepsilon=\delta-2\varepsilon.
\end{cases}
\]
This yields, if $k_1\neq k_2\in\{1, 2\}$ and $x_1, x_4\in B_{c_{k_1}},~x_2, x_3\in B_{c_{k_2}}$, we have
\[
\frac{\|x_1-x_2\|\cdot \|x_3-x_4\|}{\|x_2-x_3\|\cdot \|x_4-x_1\|}\geq \left(\frac{\delta-2\varepsilon}{2\varepsilon}\right)^2.
\]
This yields
\begin{align}
\begin{aligned}\label{E-0-2}
\mathcal{H}_p[\rho_t]&\geq \sum_{\substack{k_1, k_2=1\\ k_1\neq k_2}}^2\int_{B_\varepsilon(c_{k_1})}\int_{B_\varepsilon(c_{k_2})}\int_{B_\varepsilon(c_{k_2})}\int_{B_\varepsilon(c_{k_1})}\left(\frac{\delta-2\varepsilon}{2\varepsilon}\right)^{4p}\\
&\hspace{6cm}\times\rho_t(x_1)\rho_t(x_2)\rho_t(x_3)\rho_t(x_4)d\sigma_{1}d\sigma_{2}d\sigma_{3}d\sigma_{4}.
\end{aligned}
\end{align}
We combine the assumption \eqref{E-0-0} and calculations \eqref{E-0-1},\eqref{E-0-2} to get
\begin{align*}
\lim_{t\to\infty}\mathcal{H}_p[\rho_t]\geq\sum_{\substack{k_1, k_2=1\\ k_1\neq k_2}}^2 \left(\frac{\delta-2\varepsilon}{2\varepsilon}\right)^{4p}\delta_{k_1}^{2}\delta_{k_2}^2
=2\delta_1^2\delta_2^2\left(\frac{\delta-2\varepsilon}{2\varepsilon}\right)^{4p}.
\end{align*}
Since $\varepsilon$ can be any positive number in $(0, \varepsilon_0)$, we know that
\[
\lim_{t\to\infty}\mathcal{H}_p[\rho_t]=\infty
\]
from letting $\varepsilon\searrow0$.
\end{proof}

\begin{lemma}\label{Lem5.2}
Let $\rho_t$ be a solution of the following kinetic swarm sphere model:
\begin{align}\label{E-1}
\begin{cases}
\partial_t\rho_t+\nabla\cdot(\rho_tv_t)=0,\quad t>0,\\
\displaystyle v_t(x)=\Omega x+\int (y-\langle x, y\rangle x)\rho_t(y)d\sigma_y.
\end{cases}
\end{align}
Then we have the derivative of the order parameter 
\[
R(t)^2:=\iint\langle x, y\rangle \rho_t(x)\rho_t(y)d\sigma_xd\sigma_y
\]
as follows:
\[
\frac{d}{dt}R(t)^2=2\int\|x_c-\langle y, x_c\rangle y\|^2\rho_t(y)d\sigma_y\geq0.
\]
Furthermore, the second derivative of $R(t)^2$ is bounded.
\end{lemma}

\begin{proof}
By direct calculations, we have
\begin{align*}
\frac{d}{dt}R(t)^2&=2\int \langle x, y\rangle \rho_t(x)\partial_t\rho_t(y)d\sigma_x d\sigma_y
=-2\int \langle x, y\rangle \rho_t(x)\nabla_y\cdot(\rho_t(y) v_t(y))d\sigma_x d\sigma_y\\
&=2\int \langle v_t(y),x-\langle x, y\rangle y\rangle \rho_t(x)\rho_t(y) d\sigma_x d\sigma_y.
\end{align*}
If we define the average $x_c$ as follows:
\[
x_c(t):=\int x\rho_t(x)d\sigma_x,
\]
then we have
\[
\begin{cases}
\displaystyle v_t(x)=\Omega x+x_c-\langle x, x_c\rangle x,\\
\displaystyle\frac{d}{dt}R(t)^2=2\int \langle v_t(y), x_c-\langle x_c, y\rangle y\rangle \rho_t(y)d\sigma_y.
\end{cases}
\]
This yields
\begin{align*}
\frac{d}{dt}R(t)^2&=2\int \big\langle \Omega y+x_c-\langle y, x_c\rangle y, x_c-\langle x_c, y\rangle y\big\rangle \rho_t(y)d\sigma_y\\
&=2\int \langle \Omega y, x_c-\langle x_c, y\rangle y\rangle \rho_t(y)d\sigma_y+2\int\|x_c-\langle y, x_c\rangle y\|^2\rho_t(y)d\sigma_y\\
&=2\langle \Omega x_c, x_c\rangle+2\int\|x_c-\langle y, x_c\rangle y\|^2\rho_t(y)d\sigma_y\\
&=2\int\|x_c-\langle y, x_c\rangle y\|^2\rho_t(y)d\sigma_y\geq0.
\end{align*}
Again, we have
\begin{align*}
\frac{d^2}{dt^2}R(t)^2&=2\int\|x_c-\langle y, x_c\rangle y\|^2\partial_t\rho_t(y)d\sigma_y
=-2\int\|x_c-\langle y, x_c\rangle y\|^2\nabla_y\cdot(\rho_t(y)v_t(y))d\sigma_y.
\end{align*}
We can calculate the covariant derivative of $\|x_c-\langle y, x_c\rangle y\|^2$ with respect to $y$ as follows:
\begin{align*}
\nabla_y\|x_c-\langle y, x_c\rangle y\|^2&=\nabla_y\big\langle x_c-\langle y, x_c\rangle y, x_c-\langle y, x_c\rangle y\big\rangle\\
&=-\nabla_y\langle x_c, y\rangle^2=-2\langle x_c, y\rangle\nabla_y\langle x_c, y\rangle
=-2\langle x_c, y\rangle(x_c-\langle x_c, y\rangle y).
\end{align*}
So we have
\begin{align*}
\frac{d^2}{dt^2}R(t)^2&=4\int\langle x_c, y\rangle \langle x_c-\langle x_c, y\rangle y, v_t(y)\rangle \rho_t(y)d\sigma_y.
\end{align*}
Since $v_t$ is bounded, we know that the R.H.S. of the above relation is bounded. So we can conclude that the second derivative of the order parameter $R(t)^2$ is bounded.
\end{proof}
\begin{remark}\label{R5.1}
By Lemma \ref{Lem5.2}, we know the following limit exists:
\[
\lim_{t\to\infty}R(t)=R^\infty.
\]
Here, we used that $R(t)$ is increasing and bounded above. 
\end{remark}
\begin{lemma}[\cite{Ba} Barbalat's lemma]\label{Lem5.3}
Suppose that  a real-valued function $f: [0, \infty) \to \bbr$ is uniformly continuous and it satisfies
\[ \lim_{t \to \infty} \int_0^\top  f(s)d s \quad \textup{exists}. \]
Then, $f$ tends to zero as $t \to \infty$:
\[ \lim_{t \to \infty} f(t) = 0. \]
\end{lemma}
From Lemmas \ref{Lem5.2} and \ref{Lem5.3}, we have the following theorem.
\begin{theorem}\label{Thm5.1}
Suppose that the initial data $\rho_0$ satisfies
\[
\left\|\int y\rho_0(y)d\sigma_y\right\|\neq0,\quad i.e. \quad R(0)>0,
\]
and let $\rho_t$ be a solution of system \eqref{E-1}. Then, there exists a real number $R^\infty\in(0, 1]$ and curve $\gamma: [0, \infty)\to\bbs^d$ defined as 
\[
R^\infty:=\lim_{t\to\infty}R(t),\quad \gamma(t):=\frac{x_c}{\|x_c\|}\in\bbs^d
\]
such that
\[
\lim_{t\to\infty}\int_{B_\varepsilon(\gamma)}\rho_t(y)d\sigma_y=\frac{1}{2}(1+R^\infty),\quad \lim_{t\to\infty}\int_{B_\varepsilon(-\gamma)}\rho_t(y)d\sigma_y=\frac{1}{2}(1-R^\infty)
\]
for all $\varepsilon\in(0, \sqrt{2})$.
\end{theorem}

\begin{proof}
From Lemmas \ref{Lem5.2} and \ref{Lem5.3}, we have
\[
\lim_{t\to\infty}\frac{d}{dt}R(t)^2=0,
\]
and this yields
\[
\lim_{t\to\infty}\int\|x_c-\langle y, x_c\rangle y\|^2\rho_t(y)d\sigma_y=0.
\]
Since $R(t)=\|x_c(t)\|$ is increasing, if $R(0)>0$, then $R(t)>0$ for all $t\geq0$. From the property of the limit, we have
\[
\lim_{t\to\infty}\int\|x_c-\langle y, x_c\rangle y\|^2\rho_t(y)d\sigma_y=\lim_{t\to\infty}R(t)^2\cdot \lim_{t\to\infty}\int\|\gamma-\langle y, \gamma\rangle y\|^2\rho_t(y)d\sigma_y,
\]
and this implies
\begin{align}\label{E-2}
0=\lim_{t\to\infty}\int\|\gamma-\langle y, \gamma\rangle y\|^2\rho_t(y)d\sigma_y=\lim_{t\to\infty}\int(1-\langle y, \gamma\rangle^2)\rho_t(y)d\sigma_y.
\end{align}
For $\varepsilon\in(0, \sqrt{2})$ we have the following calculation:
\begin{align}
\begin{aligned}\label{E-3}
\int_{\bbs^d}(1-\langle y, \gamma\rangle^2)\rho_t(y)d\sigma_y
&\geq \int_{\bbs^{n}\backslash(B_\varepsilon(\gamma)\cup B_\varepsilon(-\gamma))}(1-\langle y, \gamma\rangle^2)\rho_t(y)d\sigma_y\\
&\geq \varepsilon^2\left(1-\frac{\varepsilon^2}{4}\right) \int_{\bbs^{n}\backslash(B_\varepsilon(\gamma)\cup B_\varepsilon(-\gamma))}\rho_t(y)d\sigma_y.
\end{aligned}
\end{align}
Here, we used the simple fact:
\[
\sup_{x\in\bbs^{n}\backslash(B_\varepsilon(\gamma)\cup B_\varepsilon(-\gamma))}\langle \gamma, x\rangle=1-\frac{\varepsilon^2}{2},
\]
in the last inequality. Now we combine two relations \eqref{E-2} and \eqref{E-3} to get
\[
\varepsilon^2\left(1-\frac{\varepsilon^2}{4}\right)\cdot\lim_{t\to\infty}\left(1-\int_{B_\varepsilon(\gamma)}\rho_t(y)d\sigma_y-\int_{B_\varepsilon(-\gamma)}\rho_t(y)d\sigma_y\right)=0.
\]
Since $\varepsilon\in(0, \sqrt{2})$ is a fixed number, the above relation yields
\[
\lim_{t\to\infty}\left(\int_{B_\varepsilon(\gamma)}\rho_t(y)d\sigma_y+\int_{B_\varepsilon(-\gamma)}\rho_t(y)d\sigma_y\right)=1
\]
or equivalently
\begin{align}\label{E-3-1}
\lim_{t\to\infty}\int_{\bbs^d\backslash(B_\varepsilon(\gamma)\cup B_\varepsilon(-\gamma))}\rho_t(y)d\sigma_y=0.
\end{align}

Relation \eqref{E-2} can be rewritten as
\begin{align}\label{E-4}
1=\lim_{t\to\infty}\int_{\bbs^d} \langle y, \gamma\rangle^2\rho_t(y)d\sigma_y.
\end{align}
Since $|\langle y, \gamma(t)\rangle|\leq1$, we have the following inequality:
\begin{align}\label{E-5}
1=\int_{\bbs^d}  \rho_t(y)d\sigma_y\geq\int_{\bbs^d} |\langle y, \gamma\rangle| \rho_t(y)d\sigma_y\geq\int_{\bbs^d} \langle y, \gamma\rangle^2\rho_t(y)d\sigma_y.
\end{align}
We apply the sandwich theorem to \eqref{E-5} with \eqref{E-4}, then we get
\[
\lim_{t\to\infty}\int_{\bbs^d}|\langle y, \gamma\rangle|\rho_t(y)d\sigma_y=1.
\]
From the above relation, we have
\begin{align*}
1&=\lim_{t\to\infty}\int_{\bbs^d}|\langle y, \gamma\rangle|\rho_t(y)d\sigma_y\\
&=\lim_{t\to\infty}\left(\int_{B_\epsilon(\gamma)}|\langle y, \gamma\rangle|\rho_t(y)d\sigma_y+\int_{B_\epsilon(-\gamma)}|\langle y, \gamma\rangle|\rho_t(y)d\sigma_y+\int_{\bbs^d\backslash(B_\epsilon(\gamma)\cup B_\epsilon(-\gamma))}|\langle y, \gamma\rangle|\rho_t(y)d\sigma_y\right)\\
&=\lim_{t\to\infty}\left(\int_{B_\epsilon(\gamma)}\langle y, \gamma\rangle\rho_t(y)d\sigma_y-\int_{B_\epsilon(-\gamma)}\langle y, \gamma\rangle\rho_t(y)d\sigma_y\right).
\end{align*}
Here, we used \eqref{E-3-1} in the last equality. So we get
\begin{align}\label{E-6}
\lim_{t\to\infty}\left(\int_{B_\varepsilon(\gamma)}\langle y, \gamma\rangle\rho_t(y)d\sigma_y-\int_{B_\varepsilon(-\gamma)}\langle y, \gamma\rangle\rho_t(y)d\sigma_y\right)=1.
\end{align}
On the other hand, we have
\[
R\gamma=x_c=\int_{\bbs^d} y\rho(y)d\sigma_y.
\]
We take inner product with $\gamma$ to both term of the above relation to get
\[
R=\int_{\bbs^d} \langle y, \gamma\rangle \rho_t(y)d\sigma_y.
\]
Since $\lim_{t\to\infty}R(t)$ exists, we have
\begin{align}\label{E-7}
R^\infty=\lim_{t\to\infty}\left(\int_{B_\varepsilon(\gamma)}\langle y, \gamma\rangle \rho_t(y)d\sigma_y+\int_{B_\varepsilon(-\gamma)}\langle y, \gamma\rangle \rho_t(y)d\sigma_y\right).
\end{align}
From two limits \eqref{E-6} and \eqref{E-7}, we get
\[
\lim_{t\to\infty}\int_{B_\varepsilon(\gamma)}\langle y, \gamma\rangle\rho_t(y)d\sigma_y=\frac{1}{2}(1+R^\infty),\quad \lim_{t\to\infty}\int_{B_\varepsilon(-\gamma)}\langle y, \gamma\rangle\rho_t(y)d\sigma_y=-\frac{1}{2}(1-R^\infty),
\]
and these relatoins yield
\[
\lim_{t\to\infty}\int_{B_\varepsilon(\gamma)}\rho_t(y)d\sigma_y=\frac{1}{2}(1+R^\infty),\quad \lim_{t\to\infty}\int_{B_\varepsilon(-\gamma)}\rho_t(y)d\sigma_y=\frac{1}{2}(1-R^\infty).
\]
\end{proof}

From Lemma \ref{Lem5.1} and Theorem \ref{Thm5.1}, we can obtain the following theorem.

\begin{theorem}\label{thm5.2}
Let $\rho_t$ be a solution of system \eqref{E-1} with $R(0)>0$. Then we have
\[
R^\infty=1.
\]

\end{theorem}

\begin{proof}
Since $\mathcal{H}_p[\rho_0]<\infty$, we know that $\lim_{t\to\infty}\mathcal{H}_p[\rho_t]$ can not be infinite. If $0<R^\infty<1$, then Lemma \ref{Lem5.1} and Theorem \ref{Thm5.1} yield
\[
\lim_{t\to\infty}\mathcal{H}_p[\rho_t]=\infty.
\]
This makes contradiction. Since $R(t)$ is increasing, we can conclude that $R^\infty=1$.
\end{proof}

We can compare the result of Theorem \ref{thm5.2} with the particle model case. First, we consider a solution $\{x_j\}$ of the Lohe sphere model \eqref{B-(3)-1}. Suppose that $x_j^0\neq x_k^0$ for all $j\neq k\in\mathcal{N}$. Here we assume that the number of particles $N$ is greater or equal than four. From a similar argument that we used in Theorem \ref{Thm5.1}, we know that the state converges to bipolar states. i.e. there exists a smooth curve $\gamma:[0, \infty)\to\bbs^d$ such that
\[
\text{either}\quad \lim_{t\to\infty}\|x_j(t)-\gamma(t)\|=0\quad\text{or}\quad  \lim_{t\to\infty}\|x_j(t)+\gamma(t)\|=0\quad\forall~j\in\mathcal{N}.
\]
So, there exists a partition of $\mathcal{N}=\mathcal{N}_+\sqcup\mathcal{N}_-$ such that
\[
\begin{cases}
\displaystyle\lim_{t\to\infty}\|x_j(t)-\gamma(t)\|=0,\quad\forall j\in\mathcal{N}_+,\vspace{0.2cm}\\
\displaystyle\lim_{t\to\infty}\|x_k(t)+\gamma(t)\|=0,\quad\forall k\in\mathcal{N}_-.
\end{cases}
\]
Now we assume that both $\mathcal{N}_+$ and $\mathcal{N}_-$ are greater than two. Then we can choose
\[
i_1\neq i_2\in \mathcal{N}_+,\quad i_3\neq i_4\in\mathcal{N}_-.
\] 
We check that $\mathcal{C}_{i_1i_2i_3i_4}$ is a constant of motion in Section \ref{sec:2}. However, we have
\[
\lim_{t\to\infty}\mathcal{C}_{i_1i_2i_3i_4}(t)=\lim_{t\to\infty}\frac{\|x_{i_1}(t)-x_{i_2}(t)\|^2\cdot\|x_{i_3}(t)-x_{i_4}(t)\|^2}{\|x_{i_2}(t)-x_{i_3}(t)\|^2\cdot\|x_{i_4}(t)-x_{i_1}(t)\|^2}=\infty,
\]
since 
\[
\lim_{t\to\infty}\|x_{i_2}-x_{i_3}\|=\lim_{t\to\infty}\|x_{i_4}-x_{i_1}\|=0,\quad \lim_{t\to\infty}\|x_{i_1}-x_{i_2}\|=\lim_{t\to\infty}\|x_{i_3}-x_{i_4}\|=2.
\]
So we can conclude that
\[
\min(|\mathcal{N}_+|,|\mathcal{N}_-|)\leq 1,
\]
where $|\mathcal{N}_+|$ and $|\mathcal{N}_-|$ are the number of elements of $\mathcal{N}_+$ and $\mathcal{N}_-$, respectively. If we approach this result as a density view point, then we have
\[
\min\left(\frac{|\mathcal{N}_+|}{N},\frac{|\mathcal{N}_-|}{N}\right)\leq \frac{1}{N}.
\]
Since the meaning of a mean field formulation is letting $N\to\infty$, we have
\[
0\leq \lim_{N\to\infty}\min\left(\frac{|\mathcal{N}_+|}{N},\frac{|\mathcal{N}_-|}{N}\right)\leq \lim_{N\to\infty}\frac{1}{N}=0
\]
and this yields
\[
 \lim_{N\to\infty}\min\left(\frac{|\mathcal{N}_+|}{N},\frac{|\mathcal{N}_-|}{N}\right)=0.
\]
Finally, this result is consistent with the result of Theorem \ref{thm5.2}.

\section{Nonidentical case}\label{sec:6}
\setcounter{equation}{0}

In this section, we consider non-identical models. A kinetic version of the identical swarm sphere model \eqref{B-(3)-1} is \eqref{C-(3)}. In this case, natural frequencies of all particles are equal. From now on, we study non-identical case. Generally, if natural frequencies of the swarm sphere model is not identical, we can express the system as
\begin{align}\label{F-(1)}
\begin{cases}
\displaystyle\dot{x}_i=\Omega_ix_i+\frac{\kappa}{N}(x_k-\langle x_i, x_k\rangle x_i),\quad t>0,\\
x_i(0)=x_i^0\in\bbs^d,\quad i\in\mathcal{N},
\end{cases}
\end{align}
where $\Omega_i\in\mathrm{Skew}_{d+1}(\bbr)$, and the kinetic version of system \eqref{F-(1)} can be expressed as follows:
\begin{align}\label{F-(2)}
\begin{cases}
\partial_tf_t+\nabla\cdot(L[f_t]f_t)=0,\quad f_t\in\mathcal{P}(\Xi),\quad t>0,\\
\displaystyle L[f](x, \Omega)=\Omega x+\int_{\Xi}(y-\langle y, x\rangle)f(y, \Omega_*)d\sigma_yd\Omega_*,
\end{cases}
\end{align}
where 
\[
\Xi=\bbs^d\times \mathrm{Skew}_{d+1}(\bbr)\quad\text{and}\quad \mathrm{Skew}_{d+1}(\bbr)=\left\{\Omega\in\bbr^{(d+1)\times (d+1)}: \Omega^\top =-\Omega\right\}.
\]

From \eqref{F-(1)} and \eqref{F-(2)}, we have the following non-identical version of system \eqref{B-(3)} as follows:
\begin{align}\label{F-(3)}
\begin{cases}
\displaystyle\dot{x}_i=\Omega_ix_i+X-\langle x_i, X\rangle x_i,\quad t>0,\\
x_i(0)=x_i^0\in\bbs^d,\quad i\in\mathcal{N},
\end{cases}
\end{align}
and its kinetic version is
\begin{align}\label{F-(4)}
\begin{cases}
\partial_tf_t+\nabla\cdot(u_tf_t)=0,\quad f_t\in\mathcal{P}(\Xi),\quad t>0,\\
\displaystyle u_t(x, \Omega)=\Omega x+X(t)-\langle X(t), x\rangle x.
\end{cases}
\end{align}

From straightforward calculations, we have the following proposition.
\begin{proposition}\label{Prop6.1}
Let $\mathcal{X}=\{x_j\}$ be a solution of system \eqref{F-(3)}.
For any indices $(i_1, i_2, i_3, i_4)\in C_2$ which satisfies $\Omega_{i_1}=\Omega_{i_2}=\Omega_{i_3}=\Omega_{i_4}$, a cross-ratio $\mathcal{C}_{i_1i_2i_3i_4}$ is a constant of motion.
\end{proposition}
If $\Omega_{i_a}\neq\Omega_{i_b}$ for some $a, b\in\{1, 2, 3, 4\}$, then $\mathcal{C}_{i_1i_2i_3i_4}$ can not be a constant of motion. So we need to define a subset of an index set $C_{2k}$ as follows:
\[
C_{k}^\Omega=\{(i_1, i_2, \cdots, i_{2k})\in C_{2k}: \Omega_{i_1}=\Omega_{i_2}=\cdots=\Omega_{i_{2k}}=\Omega\}.
\]
Then Proposition \ref{Prop6.1} implies that, if $(i_1, i_2, i_3, i_4)\in \bigcup_{\Omega\in \mathrm{Skew}_{d+1}(\bbr)}C_{2}^\Omega$, then $\mathcal{C}_{i_1i_2i_3i_4}$ is a constant of motion. From a similar argument we used in Corollary \ref{Cor2.1}, we have the following corollary.
\begin{corollary}\label{Coro6.1}
Let $\mathcal{X}=\{x_j\}$ be a solution of system \eqref{F-(3)}. For any $\Omega\in\mathrm{Skew}_{d+1}(\bbr)$ and $\Lambda=(i_1, i_2, \cdots, i_{2k})\in \bigcup_{\Omega\in \mathrm{Skew}_{d+1}(\bbr)}C_{k}^\Omega$,  $\mathcal{C}_\Lambda$ is a constant of motion. i.e.
\[
\frac{d}{dt}\mathcal{C}_\Lambda(\mathcal{X}(t))=0.
\]
\end{corollary}

Now, we provide a kinetic version of Proposition \ref{Prop6.1} and Corollary \ref{Coro6.1}. To obtain candidates for constants of motion, we also use a formal derivation. Since $\mathcal{C}_\Lambda$ is a constant of motion for all $\Lambda\in C_k^\Omega$ and $\Omega\in\mathrm{Skew}_{d+1}(\bbr)$, we know that $\mathcal{C}_\Lambda^p$ is also a constant of motion for all $p\in\bbr$. Therefore, the average of $\mathcal{C}_\Lambda^p$ for all $\Lambda\in C_k^\Omega$:
\[
\frac{1}{|C_k^\Omega|}\sum_{\Lambda\in C_k^\Omega}\mathcal{C}_\Lambda^p
\]
is also a constant of motion. 

We can normalize a probability density function $f_t(\cdot, \Omega)$ for fixed $\Omega\in\mathrm{Skew}_{d+1}(\bbr)$ as follows:
\[
\rho_t^\Omega(x):=\frac{f_t(x, \Omega)}{\int_{\bbs^d}f_t(y, \Omega)d\sigma_y},
\]
if $\int_{\bbs^d}f_t(y, \Omega)d\sigma_y\neq0$. Furthermore, we can easily check that $\int_{\bbs^d}f_t(y, \Omega)d\sigma_y$ is a constant along the time-evolution. Then system \eqref{F-(4)} can be converted into
\begin{align}\label{F-5}
\begin{cases}
\partial_t\rho_t^\Omega+\nabla\cdot(u_t \rho_t^\Omega)=0,\\
u_t(x, \Omega)=\Omega x+X(t)-\langle X(t), x\rangle x,
\end{cases}
\end{align}
for fixed $\Omega\in\mathrm{Skew}_{d+1}(\bbr)$. From formal calculations, we have
\begin{align*}
&\frac{1}{|C_k^\Omega|}\sum_{\Lambda\in C_k^\Omega}\mathcal{C}_\Lambda^p\simeq \int_{(\bbs^d)^{2k}}\mathcal{C}_k(x_1, x_2, \cdots, x_{2k})\rho_t^\Omega(x_1)\rho_t^\Omega(x_2)\cdots \rho_t^\Omega(x_{2k})d\sigma_1 d\sigma_2\cdots d\sigma_{2k}\\
&=\left(\int_{\bbs^d}f_t(y, \Omega)d\sigma_y\right)^{-2k}\int_{(\bbs^d)^{2k}}\mathcal{C}_k(x_1, x_2, \cdots, x_{2k})f_t(x_1, \Omega)f_t(x_2, \Omega)\cdots f_t(x_{2k}, \Omega)d\sigma_1 d\sigma_2\cdots d\sigma_{2k},
\end{align*}
and we can expect that the R.H.S. of the above relation is constant of motion. Similar to Definition \ref{Def3.1}, we define the follows.
\begin{definition}\label{Def6.1}
We define functionals $\mathcal{H}_p^\Omega$ and $\mathcal{H}_{p, k}^\Omega$ for all $p\in\bbr$ and an integer $k\geq2$ as follows:
\[
\mathcal{H}_p^\Omega[f]:=\int_{(\bbs^d)^4}\mathcal{C}(x_1, x_2, x_3, x_4)^pf(x_1,  \Omega)f(x_2, \Omega)f(x_3, \Omega)f(x_4, \Omega)d\sigma_1d\sigma_2d\sigma_3d\sigma_4,
\]
and
\[
\mathcal{H}_{p, k}^\Omega[f]:=\int_{(\bbs^d)^{2k}}\mathcal{C}_k(x_1, x_2, \cdots, x_{2k})^p f(x_1, \Omega)f(x_2, \Omega)\cdots f(x_{2k}, \Omega)d\sigma_1d\sigma_2\cdots d\sigma_{2k}.
\]
\end{definition}

From Theorem \ref{Thm3.2}, we have the following corollary.
\begin{corollary}
Let $f_t$ be a solution of system \eqref{F-(4)}. If $p\in\left(-\frac{d}{2},\frac{d}{2}\right)$, then $\mathcal{H}_p^\Omega[f_t]$ and $ \mathcal{H}_{p, k}^\Omega[f_t]$ exist and they are constant of motion functionals for all integer $k\geq2$.
\end{corollary}

\begin{proof}
Suppose that $\int_{\bbs^d}f_t(y, \Omega)d\sigma_y\neq0$, and let $\rho_t^\Omega$ be a solution of \eqref{F-5}. From Theorem \ref{Thm3.2}, we know that $\mathcal{H}_p[\rho_t^\Omega]$ and $\mathcal{H}_{p, k}[\rho_t^\Omega]$ are constants of motion, if they exist. Furthermore, we know that the functionals exist if and only if $p\in\left(-\frac{d}{2},\frac{d}{2}\right)$. We use the simple fact that $\int_{\bbs^d} f_t(y, \Omega)d\sigma_y$ is a constant along time evolution. So we define a constant $\beta_\Omega$ as follows:
\[
\beta_\Omega:=\int_{\bbs^d}f_t(y, \Omega)d\sigma_y.
\]
From identities
\begin{align*}
\mathcal{H}_p^\Omega[f_t]=\beta^4_\Omega\mathcal{H}_p[\rho_t^\Omega]\quad\text{and}\quad \mathcal{H}_{p, k}^\Omega[f_t]=\beta^{2k}_\Omega\mathcal{H}_{p, k}[\rho_t^\Omega],
\end{align*}
we know that $\mathcal{H}_p^\Omega[f_t]$ and $ \mathcal{H}_{p, k}^\Omega[f_t]$ are also constants of motion for system \eqref{F-(4)}.
\end{proof}

By Theorem 3.1, if we define $M_t^\Omega$ as follows:
\begin{align*}
M_t^\Omega(x)=w^\Omega(t)+\frac{(R^\Omega(t)x+w^\Omega(t))(1-\|w^\Omega(t)\|^2)}{\|R^\Omega(t)x+w^\Omega(t)\|^2},
\end{align*}
where $w^\Omega$ and $R^\Omega$ are solution of the following system:
\begin{align*}
\begin{cases}
\displaystyle \dot{w}^\Omega=\Omega w^\Omega+\frac{1}{2}(1+\|w^\Omega\|^2)X-\langle w^\Omega, X\rangle w^\Omega,\quad w^\Omega(0)=0,\\
\dot{R}^\Omega=\Omega+X^\Omega(w^\Omega)^\top R^\Omega-w^\Omega X^\top R^\Omega,\quad 
R^\Omega(0)=I_{d+1},
\end{cases}
\end{align*}
then we have
\[
\rho_t^\Omega=M_t^\Omega\#\rho_t^\Omega\quad\Longleftrightarrow\quad f_t(\cdot, \Omega)=M_t^\Omega\#f_0(\cdot,\Omega).
\]
This yields, if we can define a map $N_t:\Xi\to\Xi$ as follows:
\begin{align}\label{F-7}
N_t(x, \Omega)=(M_t^\Omega(x), \Omega),
\end{align}
then we have
\[
f_t=N_t\#f_0.
\]
Finally, we can obtain the WS-transform of non-identical vector models, and we can summarize the above results as following theorem.
\begin{theorem}\label{Thm6.1}
Let $f_t$ be a solution of system \eqref{F-(4)} with the initial data $f_0$. Then $f_t=N_t\#f_0$, where $N_t$ is defined in \eqref{F-7}.
\end{theorem}

\section{Conclusion}\label{sec:7}
\setcounter{equation}{0}

Throughout this paper, we have studied the Watanabe-Strogatz transformation on kinetic sphere models and its application. The Watanabe-Strogatz transformation \cite{Wa-St-1, Wa-St-2} was originally defined on the Kuramoto model. Lohe generalized this result to sphere models \cite{Lo}. In this paper, we study a kinetic version of the Watanabe-Strogatz transformation on sphere models. From this transformation, we also provided constant of motion functionals $\mathcal{H}_p$, $\mathcal{H}_{p, k}$ and necessary and sufficient conditions for well-definedness of functionals. As an application of constant of motion functionals, we proved the instability of the bipolar state of the kinetic swarm sphere model. Since the structure of the model, we only applied the constant of motion functional to the kinetic swarm sphere model. Applying constant of motion functionals to other vector models can be nice future work. We also provide the WS-transform on non-identical sphere models. We did not consider the matrix model in this paper. However, from the results in \cite{Lo-1}, we can expect that the WS-transform can be applied to kinetic matrix Riccati systems, and we leave this topic for future work.



\end{document}